\documentclass[a4paper,11pt]{article}
\usepackage[utf8]{inputenc}
\usepackage{amsthm}
\usepackage{amsmath}
\usepackage{amssymb}
\usepackage{caption}
\usepackage{cite}
\usepackage{tikz}
\usepackage{subcaption}
\usepackage[labelformat=simple]{subcaption}
\usepackage{a4wide}
\usetikzlibrary{calc}

\usepackage{changes}

\usepackage{hyperref}
\newcommand{\footremember}[2]{%
	\footnote{#2}
	\newcounter{#1}
	\setcounter{#1}{\value{footnote}}%
}
\newcommand{\footrecall}[1]{%
	\footnotemark[\value{#1}]%
}

\usepackage{enumitem}
\setlist[enumerate]{noitemsep,topsep=3pt}
\setlist[itemize]{noitemsep,topsep=3pt}

\theoremstyle{definition}
\newtheorem{defin}{Definition}

\theoremstyle{plain}
\newtheorem{lem}[defin]{Lemma}
\newtheorem{thm}[defin]{Theorem}
\newtheorem{cor}[defin]{Corollary}
\newtheorem{obs}[defin]{Observation}
\newtheorem{oldthm}{Theorem}

\newtheorem{oldlem}[oldthm]{Lemma}

\newcommand{\A}{\mathcal{A}}
\newcommand{\B}{\mathcal{B}}
\newcommand{\C}{\mathcal{C}}

 \newtheorem{lemma}[defin]{Lemma}
 \newtheorem{theorem}[defin]{Theorem}
 \newtheorem{proposition}[defin]{Proposition}
 \newtheorem{corollary}[defin]{Corollary}

 \newtheorem{observation}[defin]{Observation}
 \theoremstyle{definition}
 \newtheorem{example}[defin]{Example}

\date{June 24, 2022}

\title{On the unicyclic graphs having vertices that belong to all their (strong) metric bases}

\author{%
	Anni Hakanen \footnote{Corresponding author, email: ahakanen@outlook.com.}  \footremember{UA}{Université Clermont-Auvergne, CNRS, Mines de Saint-Étienne, Clermont-Auvergne-INP, LIMOS, 63000 Clermont-Ferrand, France} \footremember{TY}{Department of Mathematics and Statistics, University of Turku, FI-20014, Finland}%
	\and Ville Junnila \footrecall{TY}%
	\and Tero Laihonen \footrecall{TY}%
	\and Ismael G. Yero \footremember{UC}{Department of Mathematics, Universidad de C\'{a}diz, Av. Ram\'on Puyol s/n, 11202 Algeciras, Spain}
}

\begin{document}

\maketitle

\begin{abstract}
A metric basis in a graph $G$ is a smallest possible set $S$ of vertices of $G$, with the property that any two vertices of $G$ are uniquely recognized by using a vector of distances to the vertices in $S$. A strong metric basis is a variant of metric basis that represents a smallest possible set $S'$ of vertices of $G$ such that any two vertices $x,y$ of $G$ are uniquely recognized by a vertex $v\in S'$ by using either a shortest $x-v$ path that contains $y$, or a shortest $y-v$ path that contains $x$. Given a graph $G$, there exist sometimes some vertices of $G$ such that they forcedly belong to every metric basis or to every strong metric basis of $G$. Such vertices are called (resp. strong) basis forced vertices in $G$. It is natural to consider finding them, in order to find a (strong) metric basis in a graph. However, deciding about the existence of these vertices in arbitrary graphs is in general an NP-hard problem, which makes desirable the problem of searching for (strong) basis forced vertices in special graph classes. This article centers the attention in the class of unicyclic graphs. It is known that a unicyclic graph can have at most two basis forced vertices. In this sense, several results aimed to classify the unicyclic graphs according to the number of basis forced vertices they have are given in this work. On the other hand, with respect to the strong metric bases, it is proved in this work that unicyclic graphs can have as many strong basis forced vertices as we would require. Moreover, some characterizations of the unicyclic graphs concerning the existence or not of such vertices are given in the exposition as well.
\end{abstract}


\noindent
{\bf AMS Subj.\ Class.\ (2020)}: 05C12

\section{Introduction}

Resolving sets and metric bases of graphs, as well as their diverse variants, are well known in the literature due to their properties of uniquely identifying the vertices of the graph, by means of distance vectors to the vertices in such structures. Accordingly, their applications in other areas of science cover a wide range of location or identification issues in fields like chemistry, social sciences, computer sciences, and biology, among other ones. For instance, the recent article \cite{tillquist-2019} presented an interesting relationship between one of these related structures and the representation of genomic sequences. More information on theoretical results, applications, and open questions in the area can be found in the fairly complete survey \cite{till-2022+}.

The metric bases of a graph are those resolving sets which have the smallest possible number of elements. Thus, while developing some application of these structures, the use of metric bases is usually requested, since it is natural to desire optimality in the solution. However, as one can suspect, finding the metric bases of a graph is a difficult problem in general, and so, we are required to find possible tools that could help us to construct metric bases for a given graph. Several approaches to this task are known in the literature. One of such tools consist of detecting some ``key vertices'' that are always required to be part of a metric basis. In \cite{BasisForced}, such vertices were called {\em basis forced vertices}. Two variations of this idea, while considering the $\ell$-solid resolving sets and the $\{\ell\}$-resolving sets instead of the metric basis, were considered earlier in \cite{Hakanen-solid,Hakanen-DMGT}, respectively.

The idea of detecting basis forced vertices in a graph significantly contributes to having some metric bases in the graph. In order to detect them, several issues might be taken into account. First, one would be interested in knowing on which graphs have basis forced vertices, \emph{i.e.}, to know (in advance) about families of graphs in which there are or there are not such vertices. For instance, cycles, complete bipartite graphs, and trees have no basis forced vertices. Second, it would be also of interest to know several structural properties of the graphs having (or not having) basis forced vertices, that is, describing properties of such graphs, like for instance, the maximum (or minimum) possible degree, order, size, diameter, etc. Third, for those graphs having basis forced vertices, to compute the exact value or at least to bound the number of such vertices. All these issues were already discussed in \cite{BasisForced} for general graphs. Unfortunately, it was also proved in \cite{BasisForced} that deciding whether a given vertex of a graph is a basis forced vertex belongs to the class of NP-hard problems, which is indeed a big trouble since then one cannot manage to have an algorithmic solution for an arbitrary graph in connection with our purposes.

This implies that researches need to focus, among other approaches, on the investigation for special graphs classes with the goal of dealing with the three aims described above. If we think about going from the sparsest graphs to the densest ones in order to detect the existence of basis forced vertices, then we begin with trees. But they have no basis forced vertices. However, if we just add one edge to a tree, then we get a unicyclic graph, which already can have basis forced vertices as first shown in \cite{BasisForced}. One positive fact in this case is that unicyclic graphs can have at most two basis forced vertices, which makes the work more tractable. In this sense, our first aim in this work is to describe those unicyclic graphs that have either 0, 1, or 2 basis forced vertices.

On the other hand, in the investigation we also consider a variation of the metric bases (called strong metric bases). The study shows that the behavior in the existence of vertices that belong to every strong metric basis (these are called strong basis forced vertices) in unicyclic graphs drastically change with respect to the classical metric bases. That is, while unicyclic graphs can have at most two basis forced vertices, there are unicyclic graphs that can have as many strong basis forced vertices as we would require.

The two following subsections are centered into giving formal definitions concerning metric basis and strong metric basis in graphs that are necessary in our exposition.

\subsection{Classical metric dimension}

Given a simple and connected graph $G$, a vertex $x\in V(G)$ {\em resolves} or {\em identifies} two vertices $u,v\in V(G)$ if $d_G(v,x)\ne d_G(u,x)$, where $d_G(y,z)$ (or simply $d(y,z)$ if $G$ is clear from the context) stands for the distance between $y$ and $z$, \emph{i.e.}, the length of a shortest $y-z$ path in $G$. It is also said that $u,v$ are {\em resolved} or {\em identified} by $x$. A set $S\subseteq V(G)$ is a \emph{resolving set} for $G$ if every two vertices of $G$ are resolved by a vertex of $S$. A resolving set of the smallest possible cardinality in $G$ is a \emph{metric basis} and its cardinality is known as the \emph{metric dimension} of $G$, denoted by $\dim(G)$.

The concepts above were independently introduced a few decades ago in \cite{Harary76,Slater75}. The interest in this topic has exploded over the last two decades, and for instance MathSciNet database lists nowadays about 280 entries to a query with the terms ``metric dimension'' and ``graphs'', from which about 270 were published after year 2000. Some significant and recent works on this topic are for instance \cite{Claverol,Mashkaria,SedlarUnicyclic,Sedlar22,Sedlar21,Wu}. Moreover, for more information on this area we suggest the very interesting (although not yet formally published in a journal) survey \cite{till-2022+}.

A vertex $v$ of a graph $G$ is said to be a \emph{basis forced vertex} if $v$ belongs to every metric basis of $G$. Basis forced vertices were first introduced in \cite{BasisForced}, although one could say they have some antecedents in the works \cite{BagheriUnique16,BuczkowskiUnique} where graphs with a unique metric basis were considered.

\subsection{Strong metric dimension}

The concepts of strong resolving sets and strong metric dimension were introduced in connection to uniquely distinguish graphs in the following sense. In the article \cite{Sebo-2004}, the following situation was pointed out: ``\emph{For a given resolving set $T$ of a graph $H$, whenever $H$ is a subgraph of a graph $G$ and the metric vectors of the vertices of $H$ relative to $T$ agree in both $H$ and $G$, is $H$ an isometric subgraph of $G$}? In connection with this, the authors claimed that although the vectors of distances with respect to a resolving set of a graph distinguish every pair of vertices in the graph, they do not uniquely recognize all distances between vertices in the graph. In connection with these situations, the strong version of resolving sets was presented in \cite{Sebo-2004}.

For a connected graph $G$, a vertex $w\in V(G)$ \emph{strongly resolves} two different vertices $u,v\in V(G)$ if
$d_G(w,u)=d_G(w,v)+d_G(v,u)$ or $d_G(w,v)=d_G(w,u)+d_G(u,v)$. Equivalently, there is some shortest $w-u$ path
that contains $v$ or some shortest $w-v$ path containing $u$. A set $S\subseteq V(G)$ is a \emph{strong resolving set} for $G$, if every two vertices of $G$ are strongly resolved by some vertex of $S$. The cardinality of a smallest strong resolving set for $G$ is called the \emph{strong metric dimension} of $G$, denoted by $\dim_s(G)$. A \emph{strong metric basis} of $G$ is a strong resolving set of cardinality $\dim_s(G)$.

The parameter above was further related to the classical concept of vertex covers in graphs in \cite{Oellermann-2007}. To see this, we say that a vertex $u$ of $G$ is \emph{maximally distant} from another vertex $v$ if every vertex $w\in N_G(u)$ satisfies that $d_G(v,w)\le d_G(v,u)$. The set of all vertices of $G$ that are maximally distant from some vertex of the graph is called the {\em boundary} of the graph, and is denoted by $\partial(G)$. If a vertex $u$ is maximally distant from other distinct vertex $v$, and $v$ is maximally distant from $u$, then $u$ and $v$ are known to be \emph{mutually maximally distant}, or MMD for short.

Now, for a connected graph $G$, the \emph{strong resolving graph} of $G$, denoted by $G_{SR}$, is a graph that has vertex set
$V(G_{SR})=V(G)$ and two vertices $u,v$ are adjacent in $G_{SR}$ if and only if $u$ and $v$ are mutually maximally
distant in $G$ (see \cite{Kuziak} for more information on structural properties of $G_{SR}$). Clearly, if $v\notin \partial(G)$, then $v$ is an isolated vertex in $G_{SR}$. By using this construction, the following interesting connection was proved in \cite{Oellermann-2007}, where $\alpha(G)$ represents the vertex cover number of $G$.

\begin{thm}{\em \cite{Oellermann-2007}}\label{th:strong-dim-cover}
For any connected graph $G$, a set $S\subseteq V(G)$ is a strong resolving set for $G$ if and only if $S$ is also a vertex cover for $G_{SR}$. Moreover, $\dim_s(G)=\alpha(G_{SR})$.
\end{thm}

The notion of basis forced vertices for the classical metric dimension can clearly be adapted to the strong version. That is, from now on a vertex that belongs to every strong metric basis of a graph is called a \emph{strong basis forced vertex}. Based on Theorem \ref{th:strong-dim-cover}, the existence of strong basis forced vertices can be reduced to studying vertices that belong to every vertex cover set of minimum cardinality in $G_{SR}$, or equivalently (and based on the famous Gallai's theorem relating the vertex cover and independence number) vertices that do not belong to any maximum independent set of $G_{SR}$. This shows that, in such situation, several classical topics are involved.

In connection with this last comment, we remark the following. In \cite{Boros-2002}, the \emph{core} of a graph $G$, denoted by $core(G)$, was defined to be the set of vertices of $G$ that belong to all maximum independent sets of $G$, and the corona of $G$, denoted by $corona(G)$, as the vertices that belong to some maximum independent set. It can be then readily seen that the set of vertices that belong to every minimum vertex cover set of $G$ is precisely $VC(G)=V(G)\setminus corona(G)$. Consequently, we note that the set of strong basis forced vertices of $G$ is indeed $VC(G_{SR})$, and therefore, our study can be reduced to finding the set $VC(G_{SR})$ for a given graph $G$.

According to the structure of the strong resolving graph of a graph, it can be readily seen that for instance, trees (including paths), cycles, complete graphs, complete bipartite graphs, grid graphs or torus graphs (Cartesian product of two paths or two cycles, respectively), and hypercubes do not contain strong basis forced vertices. On the contrary, in \cite{Kuziak} a graph $G$ for which $G_{SR}$ is isomorphic to a path of odd order was given. Since paths of odd order have a unique vertex cover of minimum cardinality, it is then clear that such $G$ has a unique strong metric basis, and clearly, all the vertices of such unique metric basis are strong basis forced vertices.

\subsection{Other basic terminology}

The following definitions and notations shall be used in our exposition. The set of \emph{leaves} (vertices of degree one, also called \emph{pendants}) in a graph $G$ is denoted by $N_1(G)$, and we set $n_1(G)=|N_1(G)|$. Given a set $S\subsetneq V(G)$ and two vertices $u\in S$ and $v\notin S$, we write $S[u \leftarrow v] = (S \setminus \{u\}) \cup \{v\}$. For a vertex $v\in V(G)$, the \emph{open neighborhood} $N_G(v)$ of $v$ is the set of vertices adjacent to $v$. The {\em diameter} of a graph $G$ is the largest possible distance between any two vertices of $G$. A vertex $v$ is \emph{diametral} if there exists a vertex $u$ such that $d_G(u,v)$ equals the diameter of $G$.

\section{General Results on Basis Forced Vertices}

In this section, we present a couple of results on basis forced vertices in connected graphs. By the construction of the following theorem, we obtain that if $G$ is a connected graph with some basis forced vertices, then we can construct a graph with the same basis forced vertices but with more vertices in total than $G$. Our result generalizes the one of \cite[Theorem~6]{BagheriUnique16}.

\begin{thm}
Let $G$ be a connected graph with $n$ vertices, and let $B \neq \emptyset$ be the set of basis forced vertices of $G$. Let $b \in B$, and let $v\in V(G) \setminus B$ be such that $d(b,v) = \max \{ d(b,u) \ | \ u \in V(G) \setminus B \}$. Let $P_m$ be the path $v_1 \cdots v_m$. Let $H$ be the graph with 	$V(H) = V(G) \cup V(P_m)$ and 	$E(G) = E(G) \cup E(P_m) \cup \{\{v,v_1\}\}$.
Then, the set $B$ is also the set of basis forced vertices of $H$ and $\dim(H) = \dim(G)$.
\end{thm}

\begin{proof}
	
	Let $R$ be a metric basis of $G$. Let us prove that $R$ is a resolving set of $H$. Clearly, any pair of vertices in $V(G)$ is resolved by $R$. We have $d(b,v_i) = d(b,v) + i$ for all $v_i \in V(P_m)$, and thus $d(b,v_i) \neq d(b,v_j)$ when $i \neq j$. Since $d(b,v) = \max \{ d(b,u) \ | \ u \in V(G) \setminus B \}$ and $B \subseteq R$, we have $d(b,v_i) \neq d(b,u)$ for all $u \in V(G) \setminus R$. Consequently, $R$ is a resolving set of $H$ and $\dim (H) \leq \dim (G)$.
	
	Let $S$ be a metric basis of $H$. If $S$ does not contain elements of $V(P_m)$, then $S$ is clearly a resolving set of $G$. Suppose that $v_i \in S$ for some $v_i \in V(P_m)$.
	Let $x,y \in V(G)$ be such that $d(v_i,x) \neq d(v_i,y)$. We have $d(v,x) = d(v_i,x) - i \neq d(v_i,y) - i = d(v,y)$. Thus, the set $S[v_i \leftarrow v]$ is a resolving set of $G$. Consequently, $\dim (G) \leq \dim (H)$.
	
	We have shown that $\dim (H) = \dim (G)$. To conclude the proof, we will show that $H$ has the same basis forced vertices as $G$. Since the metric bases of $G$ are metric bases of $H$, any basis forced vertex of $H$ is  a basis forced vertex of $G$. Suppose then that some $b' \in B$ is not a basis forced vertex of $H$. Then there exists a metric basis $S$ of $H$ that does not contain $b'$. The set $S$ cannot be a metric basis of $G$, since then $b'$ would not be a basis forced vertex of $G$. Consequently, the set $S$ contains some $v_i$ and, by the arguments above, the set $S' = S[v_i \leftarrow v]$ is a metric basis of $G$. Since $b'$ is a basis forced vertex of $G$, we have $b' \in S'$ and $b' = v$. However, $v$ was chosen so that $v \notin B$, a contradiction.
\end{proof}

The following theorem gives us a condition for the case when we want to transform a basis forced vertex into a pendant, \emph{i.e.} if $v$ is a basis forced vertex we can attach a pendant to $v$ and that pendant becomes a basis forced vertex of the resulting graph.

\begin{thm}\label{thm:gen:pendantforced}
	Let $G$ be a connected graph with a basis forced vertex $v$. Let $H$ be the graph obtained from $G$ by attaching a pendant $u$ to $v$. The vertex $u$ is a basis forced vertex of $H$ if and only if for every metric basis $R$ of $G$ there exists a vertex $w \in N_G(v)$ such that $d_H(r,w) = d_H(r,u)$ for all $r \in R$.
\end{thm}
\begin{proof}
	Let us first consider how the metric bases of $G$ and $H$ are related to each other.
	
	Let $R$ be any metric basis of $G$. Suppose that $R$ is not a resolving set of $H$. We will show that then the set $R[v \leftarrow u]$ is a resolving set of $H$. Since $v$ is a basis forced vertex of $G$, the vertex $v$ is contained in $R$. The vertex $u$ resolves any pair of vertices in $H$ that $v$ resolves. Moreover, the vertex $u$ resolves any pair of vertices where one vertex is $u$ itself. Thus, the set $R[v \leftarrow u]$ is a resolving set of $H$.
	
	Let $S$ be a metric basis of $H$. If $u \notin S$, then $S$ is a resolving set of $G$, because $S \subseteq V(G)$ and $S$ resolves any pair of vertices in $G$. Suppose that $u \in S$. Then the set $S[u \leftarrow v]$ is a resolving set of $G$, because the vertex $v$ resolves any pair within $V(H) \setminus \{u\} = V(G)$ that $u$ resolves.
	
	Consequently, we have $\dim (G) = \dim (H)$, and the metric bases of $G$ and $H$ are exactly the same except that we may need to replace $v$ with $u$ or vice versa.
	
	($\Leftarrow$) Assume that for every metric basis $R$ of $G$ there exists a vertex $w \in N_G(v)$ such that $d_H(r,w) = d_H(r,u)$ for all $r \in R$. Suppose to the contrary that the vertex $u$ is not a basis forced vertex of $H$. Then there exists a metric basis $S$ of $H$ that does not contain $u$. Now $S$ is also a metric basis of $G$. By our assumption there exists a vertex $w \in N_G(v)$ such that $d_H(s,w) = d_H(s,u)$ for all $s \in S$. This means that the vertices $w$ and $u$ are not resolved by $S$ in $H$, a contradiction.
	
	($\Rightarrow$) Assume then that there exists a metric basis $R$ of $G$ such that for all $w \in N_G(v)$ we have $d_H(r,w) \neq d_H(r,u)$ for some $r \in R$. Since $v$ is a basis forced vertex of $G$, we have $v \in R$. The set $R$ resolves all pairs $x,y$ in $H$:
	\begin{itemize}
		\item If $x,y \in V(G)$, then $R$ resolves this pair as $R$ is a metric basis of $G$.
		\item If $x=u$ and $y=v$, then $R$ resolves this pair as $v \in R$.
		\item If $x=u$ and $y \in N_G(v)$, then $R$ resolves this pair due to our assumption.
		\item If $x=u$ and $y \notin N_G[v]$, then $d_H(v,u)=1$ and $d_H(v,y) \geq 2$, and $R$ resolves this pair as $v \in R$.
	\end{itemize}
	The set $R$ is a metric basis of $H$ that does not contain $u$, and thus $u$ is not a basis forced vertex of $H$.
\end{proof}

\begin{figure}
\begin{subfigure}[b]{0.32\linewidth}
	\centering
		\begin{tikzpicture}[scale=.6]
			\draw (0,2) -- (1,3) -- (1,1) -- (0,2) -- (-1,3) -- (-1,1) -- (0,2) -- (0,0) -- (1,1) -- (-1,1) -- (0,0) (-1,3) -- (1,3);
			\draw[dashed] (-1,3) -- (-3,3);
			\draw \foreach \x in {(-3,3),(0,0),(-1,1),(1,1),(0,2)} {
				\x node[circle, draw, fill=white, inner sep=0pt, minimum width=6pt] {}
			};
			\draw \foreach \x in {(-1,3),(1,3)} {
				\x node[circle, draw, fill=black, inner sep=0pt, minimum width=6pt] {}
			};
			\draw
				(-3,2.5) node[] {$u_1$}
				(-1,3.5) node[] {$v_1$}
				(-1.6,1) node[] {$w_1$}
				(1,3.5) node[] {$r_1$};
		\end{tikzpicture}
	\caption{The graph $G$ with the added pendant $u_1$.}\label{fig:ex:addpend1}
\end{subfigure}
\hfill
\begin{subfigure}[b]{0.32\linewidth}
	\centering
		\begin{tikzpicture}[scale=.6]
			\draw (0,2) -- (1,3) -- (1,1) -- (0,2) -- (-1,3) -- (-1,1) -- (0,2) -- (0,0) -- (1,1) -- (-1,1) -- (0,0) (-1,3) -- (1,3);
			\draw (-1,3) -- (-3,3);
			\draw[dashed] (1,3) -- (3,3);
			\draw \foreach \x in {(-1,3),(3,3),(0,0),(-1,1),(1,1),(0,2)} {
				\x node[circle, draw, fill=white, inner sep=0pt, minimum width=6pt] {}
			};
			\draw \foreach \x in {(-3,3),(1,3)} {
				\x node[circle, draw, fill=black, inner sep=0pt, minimum width=6pt] {}
			};
			\draw
			(3,2.5) node[] {$u_2$}
			(1,3.5) node[] {$v_2$}
			(1.6,1) node[] {$w_2$}
			(-3,2.5) node[] {$r_2$};
		\end{tikzpicture}
	\caption{The graph $H_1$ with the added pendant $u_2$.}\label{fig:ex:addpend2}
\end{subfigure}
\hfill
\begin{subfigure}[b]{0.32\linewidth}
	\centering
		\begin{tikzpicture}[scale=.6]
			\draw (0,2) -- (1,3) -- (1,1) -- (0,2) -- (-1,3) -- (-1,1) -- (0,2) -- (0,0) -- (1,1) -- (-1,1) -- (0,0) (-1,3) -- (1,3);
			\draw (-1,3) -- (-3,3);
			\draw (1,3) -- (3,3);
			\draw \foreach \x in {(-1,3),(1,3),(0,0),(-1,1),(1,1),(0,2)} {
				\x node[circle, draw, fill=white, inner sep=0pt, minimum width=6pt] {}
			};
			\draw \foreach \x in {(-3,3),(3,3)} {
				\x node[circle, draw, fill=black, inner sep=0pt, minimum width=6pt] {}
			};
		\end{tikzpicture}
	\caption{The graph $H_2$. \vspace{\baselineskip}}\label{fig:ex:addpend3}
\end{subfigure}
\caption{An example of how Theorem~\ref{thm:gen:pendantforced} can be used.}\label{fig:ex:addpend}
\end{figure}
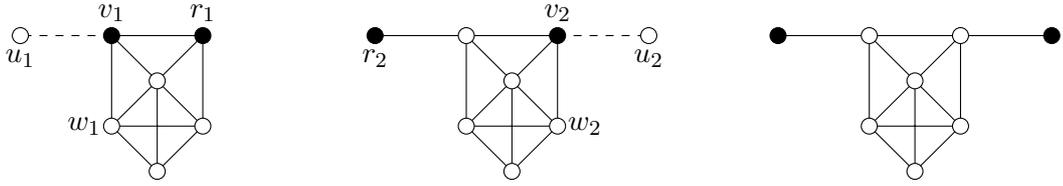

Consider the graph $G$ in Figure~\ref{fig:ex:addpend1}. This graph was shown to have a unique metric basis $\{v_1,r_1\}$ in \cite{BasisForced}. If we attach the pendant $u_1$ to the vertex $v_1$, then the vertex $u_1$ becomes a basis forced vertex of the new graph $H_1$ (illustrated in Figure~\ref{fig:ex:addpend2}). Indeed, we clearly have $d_{H_1}(v_1,u_1) = d_{H_1}(v_1,w_1) = 1$ and $d_{H_1}(r_1,u_1) = d_{H_1}(r_1,w_1) = 2$. Since the set $\{v_1,r_1\}$ is the only metric basis of $G$, the vertex $u_1$ is a basis forced vertex of $H_1$ according to Theorem~\ref{thm:gen:pendantforced}. Moreover, according to the proof of Theorem~\ref{thm:gen:pendantforced}, the set $\{u_1,r_1\}$ is the unique metric basis of $H_1$. Similarly, we can add the pendant $u_2$ to the graph $H_1$ as it is illustrated in Figure~\ref{fig:ex:addpend2}. The set $\{v_2,r_2\}$ is the unique metric basis of $H_1$, and the vertex $w_2$ fulfils the requirements of Theorem~\ref{thm:gen:pendantforced}. Thus, the pendants of the graph $H_2$ (illustrated in Figure~\ref{fig:ex:addpend3}) are both basis forced vertices, and they form the unique metric basis of $H_2$.

Let $G'$ be the graph we obtain from $G$ when we remove the edge $v_1 r_1$. The set $\{v_1,r_1\}$ is again the unique metric basis of $G'$. However, if we then attach the pendant $u_1$ to $v_1$ and thus obtain the graph $H'_1$, the vertex $u_1$ is not a basis forced vertex of $H'_1$. Indeed, now we have $d_{H'_1} (r_1,u_1) = 3$, but the vertices of $N_{G'} (v_1)$ are at distance at most $2$ from $r_1$. Thus, according to Theorem~\ref{thm:gen:pendantforced} the vertex $u_1$ is not a basis forced vertex of $H'_1$.

\section{Basis Forced Vertices in Unicyclic Graphs}

We begin this section with some preliminary information. Sedlar and \v{S}krekovski studied the metric dimension problem of unicyclic graphs in \cite{SedlarUnicyclic}, and they continued their work in \cite{Sedlar21}. We follow their terminology in order to use the characterisation of resolving sets of unicyclic graphs they showed in \cite{Sedlar21}.

Let $G$ be a unicyclic graph with the cycle $C=v_0 v_1 \cdots v_{g-1} v_0$. The components of $G - E(C)$ are denoted by $T_{v_i}$, where the subscript indicates which vertex of $C$ is contained in that component. A \emph{thread} is a path (or just a pendant). We denote by $\ell(v)$ the number of threads attached to $v$.

If $v \in V(C)$ and $\deg (v) \geq 4$ or $v \notin V(C)$ and $\deg (v) \geq 3$, then $v$ is a \emph{branching} vertex. (Notice that a vertex on the cycle that has only one thread attached to it is \emph{not} a branching vertex.)
If $T_{v_i}$ contains a branching vertex, then the vertex $v_i$ is \emph{branch-active}. The number of branch-active vertices on the cycle of $G$ is denoted by $b(G)$.
The set $S$ is \emph{branch-resolving} if for every $v\in V(G)$ of degree at least 3 the set $S$ contains a vertex from at least $\ell(v)-1$ threads attached to $v$. We denote
\[ L(G) = \sum_{v\in V(G), \ell(v) > 1} (\ell(v) - 1). \]

Let $S \subseteq V(G)$. A vertex $v_i \in V(C)$ is \emph{S-active} if $S \cap V(T_{v_i}) \neq \emptyset$. The number of $S$-active vertices on the cycle is denoted by $a(S)$. Moreover, a set $S \subseteq V(G)$ is called \emph{biactive} if $a(S) \geq 2$.
Let $v_i,v_j,v_k \in V(C)$. If $ d(v_i,v_j) + d(v_j,v_k) + d(v_k,v_i) = |V(C)|, $
then the vertices $v_i$, $v_j$ and $v_k$ form a \emph{geodesic triple} on $C$.

In \cite{SedlarUnicyclic}, the metric dimension of a unicyclic graph was determined using the concepts of branch-resolving sets and geodesic triples. The following lemma and theorem are a collection of the most important results of \cite{SedlarUnicyclic} concerning our work.

\begin{oldlem}\label{lem:Sedlar1}
\cite{SedlarUnicyclic} Let $G$ be a unicyclic graph with the cycle $C$.
\begin{enumerate}[label=\emph{(\roman*)}]
\item \label{lem:S1:resolvingbranch} If $S$ is a resolving set of $G$, then $S$ is a biactive branch-resolving set.
\item \label{lem:S1:branchsame} If $S$ is a biactive branch-resolving set of $G$, then any two vertices within the same component of $G - E(C)$ are resolved by $S$.
\item If three $S$-active vertices form a geodesic triple on the cycle $C$, then any two vertices that are in distinct components of $G - E(C)$ are resolved by $S$.
\item \label{lem:S1:branchgeodesic} Let $S$ be a branch-resolving set of $G$ with three $S$-active vertices on $C$ forming a geodesic triple. Then $S$ is a resolving set of $G$.
\end{enumerate}
\end{oldlem}

\begin{oldthm}\label{thm:unicdim}
\cite{SedlarUnicyclic} Let $G$ be a unicyclic graph. Then $\dim (G)$ is equal to $L(G) + \max \{2-b(G),0\}$ or $L(G) + \max \{2-b(G),0\} + 1$.
\end{oldthm}

In \cite{Sedlar21}, Sedlar and \v{S}krekovski continued their work on unicyclic graphs and their metric dimensions. In order to determine whether the metric dimension includes the $+1$ of Theorem \ref{thm:unicdim}, they introduced three configurations that resolving sets must avoid. These configurations will be very useful in studying the basis forced vertices of unicyclic graphs.

\begin{defin}
Let $G$ be a unicyclic graph with the cycle $C$ of length $g$ and let $S$ be a biactive branch-resolving set in $G$. We say that $C=v_0v_1 \cdots v_{g-1} v_0$ is \emph{canonically} labelled with respect to $S$ if $v_0$ is $S$-active and $k = \max\{ i \ | \ v_i \text{ is } S \text{-active}\}$ is as small as possible.
\end{defin}

\begin{defin}
Let $G$ be a unicyclic graph, and let $S$ be a biactive branch-resolving set in $G$. We say that the graph $G$ with respect to $S$ \emph{contains} configurations:
\begin{itemize}
\item[$\A$:] If $a(S)=2$, $g$ is even, and $k=g/2$.

\item[$\B$:] If $k \leq \lfloor g/2 \rfloor -1$ and there is an $S$-free thread attached to a vertex $v_i$ for some $i \in [k,\lfloor g/2 \rfloor -1] \cup [\lceil g/2 \rceil + k+1, g-1] \cup \{0\}$.

\item[$\C$:] If $a(S)=2$, $g$ is even, $k \leq g/2$ and there is an $S$-free thread of length at least $g/2-k$ attached to a vertex $v_i$ for some $i \in [0,k]$.
\end{itemize}
\end{defin}

\begin{oldthm}\label{thm:abc}\cite{Sedlar21}
Let $G$ be a unicyclic graph and let $S$ be a biactive branch-resolving set in $G$. The set $S$ is a resolving set of $G$ if and only if $G$ does not contain any of the configurations $\A$, $\B$ and $\C$ with respect to $S$.
\end{oldthm}

The following theorem follows directly from (or is a reformulation of) the results obtained in \cite{BasisForced}.

\begin{oldthm}\label{thm:basisforced}
Let $v \in V(G)$ be a basis forced vertex of a unicyclic graph $G$. Then either
\begin{enumerate}[label=\emph{(\roman*)}]
\item $v = v_i$ for some $v_i \in V(C)$ and $V(T_{v_i}) = \{v_i\}$ or
\item $v$ is a pendant attached to some $v_i \in V(C)$ and $V(T_{v_i}) = \{v_i, v\}$
\end{enumerate}
\end{oldthm}

The following result was obtained in \cite{BasisForced} as a corollary of results in \cite{PoissonUnicyclic} and \cite{SedlarUnicyclic}.

\begin{oldthm}\label{thm:basisforced2}
Let $G$ be a unicyclic graph. Then $G$ contains at most two basis forced vertices.
\end{oldthm}

A unicyclic graph with two basis forced vertices is illustrated in Figure~\ref{fig:siili}, and two unicyclic graphs with one basis forced vertex are illustrated in Figures~\ref{fig:vajaasiili} and~\ref{fig:branching}. These graphs also demonstrate that the two options for the placement of a basis forced vertex established in Theorem~\ref{thm:basisforced} are indeed possible.

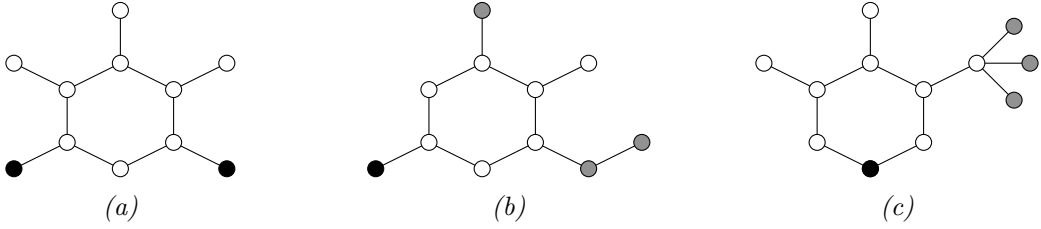
\begin{figure}
	\centering
	\begin{subfigure}[b]{.32\linewidth}
		\centering
		\begin{tikzpicture}[scale=.7]
			\draw (0,0) -- (1,.5) -- (1,1.5) -- (0,2) -- (-1,1.5) -- (-1,.5) -- (0,0);
			\draw  (1,.5) -- (2,0)
			(-1,.5) -- (-2,0)
			(1,1.5) -- (2,2)
			(-1,1.5) -- (-2,2)
			(0,2) -- (0,3);
			\draw \foreach \x in {(0,0),(1,.5),(-1,.5),(1,1.5),(-1,1.5),(0,2),(2,2),(-2,2),(0,3)} {
				\x node[circle, draw, fill=white,
				inner sep=0pt, minimum width=6pt] {}
			};
			\draw \foreach \x in {(2,0),(-2,0)} {
				\x node[circle, draw, fill=black,
				inner sep=0pt, minimum width=6pt] {}
			};
		\end{tikzpicture}
		\caption{ }\label{fig:siili}
	\end{subfigure}
	\begin{subfigure}[b]{.32\linewidth}
		\centering
		\begin{tikzpicture}[scale=.7]
			\draw (0,0) -- (1,.5) -- (1,1.5) -- (0,2) -- (-1,1.5) -- (-1,.5) -- (0,0);
			\draw  (1,.5) -- (2,0) -- (3,.5)
			(-1,.5) -- (-2,0)
			(1,1.5) -- (2,2)
			(0,2) -- (0,3);
			\draw \foreach \x in {(0,0),(1,.5),(-1,.5),(1,1.5),(-1,1.5),(0,2),(2,2)} {
				\x node[circle, draw, fill=white,
				inner sep=0pt, minimum width=6pt] {}
			};
			\draw \foreach \x in {(-2,0)} {
				\x node[circle, draw, fill=black,
				inner sep=0pt, minimum width=6pt] {}
			};
			\draw \foreach \x in {(2,0),(0,3),(3,.5)} {
				\x node[circle, draw, fill=gray!85,
				inner sep=0pt, minimum width=6pt] {}
			};
		\end{tikzpicture}
		\caption{ }\label{fig:vajaasiili}
	\end{subfigure}
	\begin{subfigure}[b]{.32\linewidth}
		\centering
		\begin{tikzpicture}[scale=.7]
			\draw (0,0) -- (1,.5) -- (1,1.5) -- (0,2) -- (-1,1.5) -- (-1,.5) -- (0,0);
			\draw (1,1.5) -- (2,2) -- (3,2)
			(2.7,1.3) -- (2,2) -- (2.7,2.7)
			(0,2) -- (0,3)
			(-1,1.5) -- (-2,2);
			\draw \foreach \x in {(0,0),(1,.5),(-1,.5),(1,1.5),(-1,1.5),(0,2),(2,2),(-2,2),(0,3)} {
				\x node[circle, draw, fill=white,
				inner sep=0pt, minimum width=6pt] {}
			};
			\draw (0,0) node[circle, draw, fill=black,
			inner sep=0pt, minimum width=6pt] {};
			\draw \foreach \x in {(3,2),(2.7,1.3),(2.7,2.7)} {
				\x node[circle, draw, fill=gray!85,
				inner sep=0pt, minimum width=6pt] {}
			};
		\end{tikzpicture}
		\caption{ }\label{fig:branching}
	\end{subfigure}
	\caption{Three examples of unicyclic graphs that contain basis forced vertices. The black vertices are basis forced vertices, the gray vertices are in some metric bases but not all, and the white vertices are not in any metric basis.}\label{fig:threetypes}
\end{figure}

\subsection{The Structure of Unicyclic Graphs With Basis Forced Vertices}

The following observation is clear by Theorems \ref{thm:unicdim} and \ref{thm:basisforced}.

\begin{obs}\label{obs:mdim}
Let $G$ be a unicyclic graph that contains $f$ basis forced vertices. If the set $R \subseteq V(G)$ is a minimum branch-resolving set of $G$, then $R$ contains no basis forced vertices. Consequently, $\dim (G) \geq L(G) + f$.
\end{obs}

Recall that $g$ is the length of the cycle $C$ in $G$, \emph{i.e.} $g$ is the girth of $G$.

\begin{thm}\label{thm:even}
Let $G$ be a unicyclic graph with at least one basis forced vertex. Then $G$ has even girth $g$.
\end{thm}
\begin{proof}
Let $G$ be a unicyclic graph with odd girth $g$. According to Theorem~\ref{thm:abc}, any biactive branch-resolving set with respect to which $G$ does not contain configuration $\B$ is a resolving set of $G$.

For a vertex $v_i$ of the cycle $C$ that is not branch-active, let $u_i \in V(T_{v_i})$ be the endvertex of the thread attached to $v_i$ if such exists or $u_i=v_i$ if no such thread exists.

Suppose that $b(G)=0$. Let $v_i$ and $v_j$ be such that $d(v_i,v_j) = \lfloor \frac{g}{2} \rfloor$. Now the set $S = \{u_i,u_j\}$ is a metric basis of $G$. Indeed, it is clearly biactive and branch-resolving. Moreover, the graph $G$ does not contain configuration $\B$ with respect to the set $S$, because if $C$ is canonically labelled with respect to the set $S$, then $k = \lfloor \frac{g}{2} \rfloor$. Thus, the set $S$ is a resolving set of $G$ due to Theorem~\ref{thm:abc}. Clearly, the graph $G$ cannot have a smaller resolving set, and thus the set $S$ is a metric basis of $G$. Since $v_i$ and $v_j$ can be chosen in multiple ways, the graph $G$ does not contain basis forced vertices.

Suppose then that $b(G) \geq 1$. Let $R$ be a branch-resolving set of cardinality $L(G)$. Let $C$ be canonically labelled with respect to the set $R$. Notice that the $R$-active vertices are exactly the branch-active vertices on the cycle. If $k \geq \lfloor \frac{g}{2} \rfloor$, then the set $R$ is a metric basis of $G$ due to Lemma~\ref{lem:Sedlar1}\ref{lem:S1:branchgeodesic}, and the graph $G$ does not contain basis forced vertices due to Observation~\ref{obs:mdim}. Suppose that $k < \lfloor \frac{g}{2} \rfloor$. Suppose to the contrary that $G$ contains at least one basis forced vertex. Now $\dim (G) \geq L(G) + 1$ due to Observation~\ref{obs:mdim}. However, now both sets $R \cup \{u_{\lfloor \frac{g}{2} \rfloor}\}$ and $R \cup \{u_{\lfloor \frac{g}{2} \rfloor + 1}\}$ are metric bases of $G$ according to Theorem~\ref{thm:abc}. The vertices $u_{\lfloor \frac{g}{2} \rfloor}$ and $u_{\lfloor \frac{g}{2} \rfloor + 1}$ are clearly not basis forced vertices, and the set $R$ cannot contain basis forced vertices due to Observation~\ref{obs:mdim}. Consequently, the graph $G$ does not contain basis forced vertices (a contradiction).

In conclusion, if the graph $G$ has odd girth, then it does not contain any basis forced vertices. Thus, if the graph $G$ contains basis forced vertices, then the girth is even.
\end{proof}

\begin{lem}\label{lem:nobfs}
Let $G$ be a unicyclic graph with $g\geq 4$. If $\dim (G) \geq L(G) + 2$ and $b(G) \geq 1$, then the graph $G$ does not contain any basis forced vertices.
\end{lem}
\begin{proof}
Let $v_0$ be branch-active and let $S$ be a branch-resolving set of $G$ of cardinality $L(G)$. Due to Theorem~\ref{thm:basisforced}, the elements of $S$ are not basis forced vertices. Let $v_i$ and $v_j$ be vertices of the cycle $C$ such that they form a geodesic triple with $v_0$. The set $S \cup \{v_i,v_j\}$ is a metric basis of $G$ according to Lemma~\ref{lem:Sedlar1}\ref{lem:S1:branchgeodesic}. Since $g \geq 4$, we can choose the vertices $v_i$ and $v_j$ in multiple ways. Thus, it is easy to see that the vertices of the cycle are not basis forced vertices.
Due to Observation~\ref{obs:mdim}, the elements of $S$ are not basis forced vertices either. Therefore, the graph $G$ contains no basis forced vertices.
\end{proof}

Lemma~\ref{lem:nobfs} implies that if a unicyclic graph has basis forced vertices, then $\dim (G) \leq L(G) + 1$ or $b(G) = 0$. If $\dim (G) \leq L(G) + 1$, then due to Observation~\ref{obs:mdim}, we have $\dim (G) = L(G) + 1$ and the graph $G$ contains exactly one basis forced vertex. Thus, if $b(G) \geq 1$, the graph $G$ can contain at most one basis forced vertex (see Figure~\ref{fig:branching}). According to Theorem~\ref{thm:basisforced2}, a unicyclic graph can contain at most two basis forced vertices, and indeed when $b(G)=0$, the graph $G$ can contain either one or two basis forced vertices (see Figures~\ref{fig:vajaasiili} and~\ref{fig:siili}, respectively).

The following two lemmas show that we can divide unicyclic graphs with basis forced vertices into the three types represented by the three example graphs introduced in Figure~\ref{fig:threetypes}.

\begin{lem}\label{lem:twoforced}
Let $G$ be a unicyclic graph with $g \geq 4$. The graph $G$ contains two basis forced vertices if and only if $b(G) = 0$, $\dim (G) = 2$, and $G$ has a unique metric basis.
\end{lem}
\begin{proof}
If the graph $G$ has a unique metric basis, then it contains exactly $\dim (G) = 2$ basis forced vertices.	

Assume that $G$ contains two basis forced vertices. According to Observation~\ref{obs:mdim}, we have $\dim (G) \geq L(G) + 2$. Now, we have $b(G) = 0$ due to Lemma~\ref{lem:nobfs}. Consequently, $L(G) = 0$. If $\dim (G) \geq 3$, then any set $S \subseteq V(G)$ such that $|S| = 3$ and the $S$-active vertices form a geodetic triple on the cycle is a metric basis of $G$ due to Lemma~\ref{lem:Sedlar1}\ref{lem:S1:branchgeodesic}. Thus, it is easy to find a metric basis of $G$ that does not contain at least one of the basis forced vertices, a contradiction. Therefore, we have $\dim (G) = 2$ and the only metric basis of $G$ consists of the two basis forced vertices.
\end{proof}

\begin{lem}\label{lem:oneforced}
Let $G$ be a unicyclic graph with $g \geq 4$. If $G$ contains exactly one basis forced vertex, then $b(G) \leq 1$.
\end{lem}
\begin{proof}
Suppose to the contrary that $b(G) \geq 2$. Let $v_i$ and $v_j$ be branch-active. Let $S$ be a metric basis of $G$ and let $v$ be a basis forced vertex of $G$. (Notice that due to Theorem~\ref{thm:basisforced}, we have $v \neq v_i,v_j$.)

Since $g$ is even due to Theorem \ref{thm:even}, there exists a vertex $v_k$ such that $v_k \neq v,v_i,v_j$ and $v_k$, $v_i$ and $v_j$ form a geodesic triple on the cycle of $G$. Due to Theorem \ref{thm:basisforced} and Lemma~\ref{lem:Sedlar1}\ref{lem:S1:resolvingbranch}, the set $S \setminus \{v\}$ is a branch-resolving set of $G$. Thus, the set $S[v \leftarrow v_k]$ is a metric basis of $G$ according to Lemma~\ref{lem:Sedlar1}\ref{lem:S1:branchgeodesic}, a contradiction.
\end{proof}

Thus, there are three types of unicyclic graphs that have basis forced vertices:
\begin{itemize}
\item We have $b(G)=0$, the graph $G$ contains two basis forced vertices and has a unique metric basis (for example, the graph in Figure~\ref{fig:siili}).

\item We have $b(G)=0$ and the graph $G$ contains exactly one basis forced vertex (for example, the graph in Figure~\ref{fig:vajaasiili}).

\item We have $b(G)=1$ and the graph contains exactly one basis forced vertex (for example, the graph in Figure~\ref{fig:branching}).
\end{itemize}

In Section \ref{sec:cycleorpendant}, we investigate whether the basis forced vertices are on the cycle or as pendants (these are the only two possibilities according to Theorem~\ref{thm:basisforced}). The remainder of this section is devoted to finding more general properties of unicyclic graphs with basis forced vertices.

\begin{thm}\label{thm:bfmdim}
Let $G$ be a unicyclic graph with $g\geq 4$ and at least one basis forced vertex. Then $\dim (G) = L(G) + \max \{2-b(G),0\}$.
\end{thm}
\begin{proof}
Recall that according to Theorem \ref{thm:unicdim} we have either $\dim (G) = L(G) + \max \{2-b(G),0\}$ or $\dim (G) = L(G) + \max \{2-b(G),0\} + 1$.

Suppose to the contrary that $\dim (G) = L(G) + \max \{2-b(G),0\} + 1$. The graph $G$ has at most two basis forced vertices according to Theorem~\ref{thm:basisforced2}. Now either $b(G) = 0$ or $b(G) = 1$ due to Lemmas~\ref{lem:twoforced} and~\ref{lem:oneforced}.

Suppose that $b(G) = 0$. Then $L(G) = 0$ and $\dim (G) = 3$. Now, according to Lemma~\ref{lem:Sedlar1}\ref{lem:S1:branchgeodesic}, any set $S \subseteq V(G)$ such that $|S|=3$ and the $S$-active vertices form a geodesic triple is a metric basis of $G$. However, we can choose the elements of the geodesic triple in multiple ways, and thus the graph $G$ cannot have any basis forced vertices.

Suppose then that $b(G) = 1$. Now $\dim (G) = L(G) + 2$, and the graph $G$ does not have any basis forced vertices according to Lemma \ref{lem:nobfs}, a contradiction.
\end{proof}

Let us consider the metric dimensions of different types of unicyclic graphs that contain basis forced vertices. For this purpose, let $G$ again be a unicyclic graph containing basis forced vertices.
\begin{itemize}
	\item If $b(G)=0$, then $L(G)=0$ and $\dim (G) = 2$ due to Theorem~\ref{thm:bfmdim}. On the first hand, if the graph $G$ contains two basis forced vertices, then $G$ has a unique metric basis as we saw before; for an example, see Figure~\ref{fig:siili}. On the other hand, if the graph $G$ contains only one basis forced vertex, then the metric bases of $G$ consist of the basis forced vertex and one other vertex. Indeed, in our example graph in Figure~\ref{fig:vajaasiili}, the metric bases of $G$ consist of one gray vertex and the basis forced vertex that is illustrated in black.
	
	\item If $b(G)=1$, then $\dim(G)= L(G) + 1$ according to Theorem~\ref{thm:bfmdim}. Furthermore, the graph $G$ contains only one basis forced vertex, and all elements of a metric basis that contribute towards $L(G)$ are in one component $T_{v_i}$ since $b(G)=1$. Thus, the metric bases of $G$ consist of the basis forced vertex and the vertices of the component $T_{v_i}$ (see Figure~\ref{fig:branching}).
\end{itemize}

Recall that according to Lemma~\ref{lem:Sedlar1}\ref{lem:S1:resolvingbranch} we have $a(S) \geq 2$ for any metric basis $S$ of a unicyclic graph. Due to the discussion above, we have the following corollary.

\begin{cor}\label{cor:active2}
Let $G$ be a unicyclic graph with at least one basis forced vertex. If $S$ is a metric basis of $G$, then $a(S) = 2$.
\end{cor}

According to Corollary~\ref{cor:active2}, if a unicyclic graph $G$ has basis forced vertices, then there are only two $S$-active vertices on the cycle for any metric basis $S$. Consequently, if a unicyclic graph $G$ contains basis forced vertices, then a canonical labelling of $C$ is always such that the vertices $v_0$ and $v_k$ are the only $S$-active vertices for any metric basis $S$. Moreover, we can choose in which component ($T_{v_0}$ or $T_{v_k}$) the basis forced vertex we are considering is located. In the majority of this paper, the component $T_{v_k}$ contains a basis forced vertex and $v_0$ is branch-active if $b(G) \neq 0$. The following lemma describes how the two $S$-active vertices are located with respect to one another, when $S$ is a metric basis of a unicyclic graph that contains basis forced vertices.

\begin{lem}\label{lem:klimits}
Let $G$ be a unicyclic graph, and let $S$ be a metric basis of $G$. If $G$ contains a basis forced vertex in $T_{v_k}$, then $2 \leq k < g/2$.
\end{lem}
\begin{proof}
Since $G$ contains at least one basis forced vertex, $a(S)=2$ due to Corollary~\ref{cor:active2}. Thus, we have $0 < k \leq g/2$ (due to the definition of canonical labelling). If $k = g/2$, then the graph $G$ contains configuration $\A$ with respect to $S$, and the set $S$ is not a resolving set of $G$ according to Theorem \ref{thm:abc}. Thus, $k < g/2$.

Suppose that $k=1$. According to Lemma~\ref{lem:twoforced} and Lemma~\ref{lem:oneforced}, we have $b(G) \leq 1$. If $C$ contains a branch-active vertex, then that vertex is $v_0$ due to Theorem~\ref{thm:basisforced}. Thus, for any $v_i$, $i \neq 0$, either $\deg (v_i) = 2$ or there is exactly one thread attached to $v_i$.

Since $S$ is a metric basis of $G$, then $G$ does not contain configuration $\B$ with respect to $S$ due to Theorem \ref{thm:abc}. Thus, there are no $S$-free threads at vertices $v_i$ where $i \in [0,g/2-1] \cup [g/2+2,g-1]$. Consequently, we have $\deg(v_i)=2$ for all $i \in [2,g/2-1] \cup [g/2+2,g-1]$. Let $u$ be $v_{g/2+1}$ or the end-vertex of the thread attached to $v_{g/2+1}$ if such a thread exists.
We claim that now the set $R=S[v \leftarrow u]$ is a metric basis of $G$. Let us relabel $C$ so that $u_0=v_0$ and $u_{g/2-1}=v_{g/2+1}$. The labelling $u_i$ is canonical with respect to $R$ with $k=g/2-1$. The graph $G$ does not contain configuration $\A$ with respect to the set $R$. There are no $R$-free threads at vertices $u_i$ where $i \in [0,g/2-1]$. Thus, the graph $G$ does not contain configuration $\B$ or $\C$ with respect to the set $R$. Thus, the set $R$ is a metric basis of $G$ according to Theorem \ref{thm:abc}, a contradiction.
\end{proof}

According to Lemmas~\ref{lem:twoforced} and~\ref{lem:oneforced}, we have $b(G) \leq 1$ for every unicyclic graph $G$ that contains basis forced vertices. Thus, there is at most one branch-active vertex on the cycle $C$. This branch-active vertex is always $S$-active for any metric basis $S$, since the set $S$ is branch-resolving due to Lemma~\ref{lem:Sedlar1}\ref{lem:S1:resolvingbranch}. Therefore, every vertex on the cycle that is not branch-active is either of degree~2 or has exactly one thread attached to it. The following two lemmas consider how long the $S$-free threads can be, and without which $S$-free threads the graph $G$ cannot have basis forced vertices.

\begin{lem}\label{lem:bfproperties}
Let $G$ be a unicyclic graph with at least one basis forced vertex. Let $S$ be a metric basis of $G$, and let $i \in [1,k-1]$ where $k$ is the index of the canonical labelling. The following properties hold.
\begin{itemize}
\item[\emph{(i)}] Either $\deg (v_i) = 2$ or there is exactly one thread at $v_i$.
\item[\emph{(ii)}] A thread at $v_i$ is of length at most $g/2-k-1$.
\item[\emph{(iii)}] There exists a thread of length $g/2-k-1$ at some $v_i$ or $k=g/2-2$ and there is no basis forced vertex on the cycle.
\item[\emph{(iv)}] For each $j \in [k+1,g/2-1] \cup [g/2+k+1,g-1]$, we have $\deg(v_j) = 2$.
\end{itemize}
\end{lem}
\begin{proof}
Let $v$ be a basis forced vertex of $G$. Let $C$ be canonically labelled so that $T_{v_k}$ contains $v$. Due to Theorem \ref{thm:even} and Lemma \ref{lem:klimits}, $g$ is even and $2 \leq k < g/2$.

(i) Otherwise $v_i$ is branch-active. Due to Corollary \ref{cor:active2}, we have $a(S)=2$, and the claim follows since a branch-active vertex must be $S$-active if $S$ is a metric basis.

(ii) According to Theorem \ref{thm:abc} the graph $G$ does not contain configuration $\C$ with respect to the set $S$, and the claim follows.

(iii) If $k=g/2-1$, then $g/2-k-1 = 0$, and the claim is trivial.

Suppose then that $2\leq k \leq g/2-2$. Since $v$ is a basis forced vertex, the set $R=S[v \leftarrow v_{k+1}]$ is not a metric basis of $G$. According to Theorem \ref{thm:abc}, the graph $G$ contains configuration $\A$, $\B$ or~$\C$ with respect to the set $R$. However, $G$ does not contain configuration
\begin{itemize}
\item[$\A$:] Since $k \leq g/2-2$, we have $k+1 \leq g/2-1$, and the graph $G$ does not contain configuration $\A$ with respect to the set $R$.

\item[$\B$:] The graph $G$ does not contain configuration $\B$ with respect to the set $S$. Thus, there does not exist a thread at any $v_i$ where $i \in [k+1,g/2-1] \cup [g/2+k+2,g-1]$. Consequently, the graph $G$ does not contain configuration $\B$ with respect to the set $R$.
\end{itemize}
Thus, $G$ contains configuration $\C$ with respect to the set $R$. In other words, there exists a thread of length at least $g/2-(k+1)$ at some $v_i$ where $i\in [1,k]$. If such a thread does not exist at any $v_i$ where $i \in [1,k-1]$, then there is a thread of length $g/2-(k+1)$ at $v_k$. Since $T_{v_k}$ contains the basis forced vertex $v$, we have $g/2-(k+1) \leq 1$ due to Theorem \ref{thm:basisforced}. Consequently, we have $k = g/2-2$ and $v$ is a pendant attached to $v_k$. 

(iv) Suppose to the contrary that $v_j \in C$ is a vertex such that $\deg(v_j) > 2$ and $j \in [k+1,g/2-1] \cup [g/2+k+1,g-1]$. By the previous observation (above Lemma~\ref{lem:bfproperties}), this implies that there is exactly one thread attached to $v_i$. Hence, the graph $G$ contains configuration~$\B$ with respect to $S$ and a contradiction follows with the fact that $S$ is a metric basis of $G$.
\end{proof}

In the following lemma, we introduce a new parameter that is very useful in characterising unicyclic graphs with basis forced vertices. Note that the vertex $v_{g/2 + j}$ is the vertex antipodal to $v_j$ on the cycle $C$.

\begin{lem}\label{lem:mthread}
Let $G$ be a unicyclic graph with at least one basis forced vertex $v$. Let $S$ be a metric basis of $G$ and let $C$ be (canonically) labelled so that $T_{v_k}$ contains the basis forced vertex $v$. Let $m = \min \{ j \geq 1 \ | \ \deg(v_j) \geq 3 \text{ or } \deg(v_{g/2+j}) \geq 3 \}$. Then, it follows $m < k$ and there exists a thread of length at least $m$ at some $v_i$ where $i \in [g/2+m+1,g/2+k]$.
\end{lem}
\begin{proof}
According to Theorem~\ref{thm:abc}, the graph $G$ does not contain configuration $\B$ with respect to the set $S$. Thus, we have $\deg(v_i) = 2$ for all $i \in [k+1,g/2-1] \cup [g/2+k+1,g-1]$.

Let us first show that $m < k$. Suppose to the contrary that $m \geq k$. Now, we have $\deg (v_i) = 2$ for all $i \in [1,k-1] \cup [g/2+1,g/2+k-1]$ due to the definition of $m$. The only vertices $v_i$ for which we may have $\deg (v_i) \geq 3$ are $v_0$, $v_k$, $v_{g/2}$ and $v_{g/2+k}$. Let $u$ be $v_{g/2+k}$ or the end-vertex of the thread attached to $v_{g/2+k}$ if such exists (if $b(G)=1$, then $v_0$ is branch-active, and thus the vertex $u$ is well-defined). The graph $G$ clearly does not contain configuration $\A$ with respect to the set $S[v \leftarrow u]$.
Neither does it contain configuration $\B$, since $\deg (v_i) = 2$ for all $i \in [1,k-1] \cup [g/2+k,g/2+k-1]$ and there are no threads without elements of $S[v \leftarrow u]$ attached to $v_0$ or $v_{g/2+k}$.
We also have $\deg (v_i) = 2$ for all $i \in [g/2+k+1,g-1]$, and thus the graph $G$ does not contain configuration $\C$ with respect to the set $S[v \leftarrow u]$.
Therefore, the set $S[v \leftarrow u]$ is a metric basis of $G$ according to Theorem~\ref{thm:abc}, a contradiction.

Suppose then that there does not exist a thread of length at least $m$ at some $v_i$ where $i \in [g/2+m+1,g/2+k]$. Let $u$ be $v_{g/2+m}$ or the end-vertex of the thread attached to $v_{g/2+m}$ if such exists. The set $S[v \leftarrow u]$ is a metric basis according to Theorem \ref{thm:abc}:
\begin{itemize}
\item[$\A$:] Since $1 \leq m < k \leq g/2-1$, the graph $G$ does not contain configuration $\A$ with respect to the set $S[v \leftarrow u]$.

\item[$\B$:] Due to the definition of $m$, we have $\deg (v_i) = 2$ for all $i \in [1,m-1] \cup [g/2+1,g/2+m-1]$. Moreover, if there exists a thread attached to $v_{g/2+m}$, then it contains an element of $S[v \leftarrow u]$. Since the set $S$ is a metric basis of $G$, there is no $S$-free thread at $v_0$. Consequently, there is no $S[v \leftarrow u]$-free thread at $v_0$ either. Thus, the graph $G$ does not contain configuration $\B$ with respect to the set $S[v \leftarrow u]$.

\item[$\C$:] Any thread attached to a vertex $v_i$ where $i \in [g/2+m+1,g/2+k]$ is of length at most $m-1 = g/2-(g/2-m)-1$ (if we label $C$ again the new $k$ would be $g-(g/2+m)=g/2-m$). Thus, the graph $G$ does not contain configuration $\C$ with respect to the set $S[v \leftarrow u]$.
\end{itemize}
Consequently, there exists a metric basis of $G$ that does not contain the basis forced vertex $v$, a contradiction.
\end{proof}

\subsection{Basis Forced Vertex on the Cycle or as a Pendant}\label{sec:cycleorpendant}

Recall that, according to Theorem~\ref{thm:basisforced}, a basis forced vertex of a unicyclic graph is either on the cycle or is a pendant that is attached to the cycle. In this section, we show that there is only a slight structural difference between unicyclic graphs that have a basis forced vertex on the cycle compared to those that have a basis forced vertex as a pendant.

The following lemma states that if $G$ is a unicyclic graph with a basis forced vertex $v_i \in V(C)$, then we can construct a graph with a pendant as a basis forced vertex simply by attaching a pendant to $v_i$. The attached pendant is a basis forced vertex of the new graph.

\begin{lem}\label{lem:pendant}
Let $G$ be a unicyclic graph and let $H$ be the graph we obtain from $G$ by attaching one pendant $u$ to a vertex $v_i$, where $i \in \{0, \ldots , g-1\}$. If the vertex $v_i$ is a basis forced vertex of $G$, then the vertex $u$ is a basis forced vertex of $H$.
\end{lem}
\begin{proof}
Let $S$ be a metric basis of $G$. Let $C$ be canonically labelled so that $v_0$ is a basis forced vertex of $G$ (recall that this is possible due to Corollary~\ref{cor:active2}). The graph $G$ has even girth $g$ due to Theorem \ref{thm:even}. According to Corollary \ref{cor:active2} and Lemma \ref{lem:klimits}, we have $a(S) = 2$ and $k < g/2$. According to Theorem \ref{thm:basisforced}, we have $V(T_{v_0}) = \{v_0\}$, and thus $d_H(s,u) = d_G(s,v_0) + 1 = d_H(s,v_{g-1})$ for all $s \in S$. The vertex $u$ is now a basis forced vertex of $H$ due to Theorem \ref{thm:gen:pendantforced}.
\end{proof}

Consider again the graph in Figure~\ref{fig:branching}. The black vertex is a basis forced vertex. As it is on the cycle, we can attach a pendant to it, and the pendant becomes a basis forced vertex of the new graph. The metric bases behave exactly as in the original graph, that is, a metric basis consists of the added pendant and two of the gray vertices.

Constructing a graph with a basis forced vertex on the cycle from a graph that has one as a pendant is not so simple. Indeed, the following lemma gives us two cases where we cannot remove the pendant and have the adjacent cycle vertex become a basis forced vertex.

\begin{lem}\label{lem:nobftocycle}
Let $G$ be a unicyclic graph with a basis forced vertex $v$ that is a pendant. Let $S$ be a metric basis of $G$, and let $C$ be canonically labelled so that $v \in V(T_{v_k})$. If
\begin{enumerate}[label={\em(\roman*)}]
\item $k=g/2-2$ and there are no threads at vertices $v_i$ where $i \in [1,k-1]$, or
\item $k=g/2-1$ and $b(G)=0$,
\end{enumerate}
then the vertex $v_k$ is not a basis forced vertex of the graph $G-v$.
\end{lem}
\begin{proof}
(i) Let $R=S[v \leftarrow v_{g/2-1}]$. The graph $G-v$ clearly does not contain configuration $\A$ with respect to the set $R$. Neither does it contain configuration $\B$, because the graph $G$ does not contain configuration $\B$ with respect to the set $S$ and thus $\deg (v_{g/2-1})=2$. Since there are no threads at vertices $v_i$ where $i \in [1,g/2-2]$ (in the graph $G-v$), the graph $G-v$ does not contain configuration $\C$ with respect to the set $R$. Thus, the set $R$ is a metric basis of $G-v$ due to Theorem~\ref{thm:abc}, and $v_k$ is not a basis forced vertex of $G-v$.

(ii) Since $S$ is a metric basis of $G$, the graph $G$ does not contain configuration $\A$, $\B$, or $\C$ with respect to the set $S$. When we remove the pendant $v$ from the graph $G$ and replace $v$ in the set $S$ with its neighbour $v_k$, the graph $G-v$ clearly does not contain configuration $\A$, $\B$, or $\C$ with respect to the set $S[v \leftarrow v_k]$. Thus, the set $S[v \leftarrow v_k]$ is a metric basis of $G-v$ according to Theorem~\ref{thm:abc}.

Since $b(G)=0$, we have $\dim (G) = 2$ due to Theorem~\ref{thm:bfmdim}. Let $R=\{u,v_1\}$, where $u$ is $v_{g/2}$ or the end-vertex of the thread attached to $v_{g/2}$ if such exists. The set $R$ is a metric basis of $G-v$ due to Theorem~\ref{thm:abc}:
\begin{itemize}
\item[$\A$:] The graph $G-v$ clearly does not contain configuration $\A$ with respect to the set $R$.

\item[$\B$:] Due to Lemma~\ref{lem:bfproperties}(ii), we have $\deg_G(v_i)=2$ for all $i \in [1,g/2-2]$. Thus, there are no $R$-free threads at $v_1$. There are no $R$-free threads at $v_{g/2}$ either, since $b(G)=0$ and if there does exist a thread at $v_{g/2}$, then it contains the vertex $u$. Consequently, the graph $G-v$ does not contain configuration $\B$ with respect to the set $R$.

\item[$\C$:] As we saw before, $\deg_{G} (v_i) = 2$ for all $i \in [1,g/2-2]$. Due to Theorem \ref{thm:basisforced}, $\deg_{G-v} (v_{g/2-1}) = 2$. Thus, the graph $G-v$ does not contain configuration $\C$ with respect to the set $R$.
\end{itemize}
Thus, there exists a metric basis of $G-v$ that does not contain $v_k$. Consequently, $v_k$ is not a basis forced vertex of $G-v$.
\end{proof}

The graphs in Figures~\ref{fig:siili} and~\ref{fig:vajaasiili} both fall under case~(ii) of Lemma~\ref{lem:nobftocycle}. Thus, if we remove one of the basis forced vertices (illustrated in black) from one of these graphs, the adjacent cycle vertices are not basis forced vertices of the resulting graphs.

In addition to the two cases presented in Lemma~\ref{lem:nobftocycle}, we have one more requirement in order to have a basis forced vertex on the cycle.

\begin{lem}\label{lem:cyclebfthread}
Let $G$ be a unicyclic graph with a basis forced vertex $v$. If $v \in V(C)$, then there is a thread at the vertex antipodal to $v$. (If $v=v_i$, then there is a thread at $v_{g/2+i}$.)
\end{lem}
\begin{proof}
Let $S$ be a metric basis of $G$, and let $v_k$ be a basis forced vertex. Now $g$ is even and $2 \leq k < g/2$ due to Theorem \ref{thm:even} and Lemma~\ref{lem:klimits}.

Suppose to the contrary that there is no thread at $v_{g/2+k}$ (\emph{i.e.} $\deg (v_{g/2+k}) = 2$). Let $u$ be $v_{k-1}$ or the end-vertex of the thread at $v_{k-1}$ if such a thread exists. Let $R=S[v_k \leftarrow u]$. Due to Theorem \ref{thm:abc}, the graph $G$ does not contain configuration $\B$ with respect to the set $S$. Thus, there are no threads at vertices $v_i$ where $i \in [k+1,g/2-1] \cup [g/2+k+1,g-1]$. Since there is no thread at $v_{g/2+k}$ by assumption and no thread at $v_k$ by Theorem \ref{thm:basisforced2}, there are no $R$-free threads at vertices $v_i$ where $i \in [k-1,g/2-1] \cup [g/2+k,g-1] \cup \{0\}$. Thus, the graph $G$ does not contain configuration $\B$ with respect to the set $R$. It is clear that the graph $G$ does not contain configuration $\A$ or $\C$ with respect to the set $R$ either. Thus, the set $R$ is a metric basis of $G$ according to Theorem \ref{thm:abc}, a contradiction.
\end{proof}

Excluding the two cases in Lemma~\ref{lem:nobftocycle}, the condition in Lemma~\ref{lem:cyclebfthread} is sufficient for the case where $b(G)=1$. Indeed, the following lemma states that when $b(G)=1$ and $G$ contains a pendant $v$ that is a basis forced vertex, we can remove the pendant $v$ and the adjacent cycle vertex becomes a basis forced vertex of the resulting graph as long as the thread described by Lemma~\ref{lem:cyclebfthread} is present and we are not considering one of the two cases of Lemma~\ref{lem:nobftocycle} (cf. Lemma~\ref{lem:bfproperties}(iii)).

\begin{lem}\label{lem:pendanttocycle}
Let $G$ be a unicyclic graph with $b(G)=1$ and one basis forced vertex $v$ that is a pendant. Let $S$ be a metric basis of $G$, and let $C$ be canonically labelled so that $v_0$ is branch-active and $v \in V(T_{v_k})$. If there exists a thread of length $g/2 - k - 1$ at some $v_i$ where $i \in [1,k-1]$, then the vertex $v_k$ is a basis forced vertex of $G-v$ if and only if there exists a thread attached to the vertex $v_{g/2+k}$ (\emph{i.e.} the vertex antipodal to $v_k$).
\end{lem}
\begin{proof}
Recall that by Lemma~\ref{lem:klimits} we have $k \leq g/2-1$. If $v_k$ is a basis forced vertex of $G-v$, then there exists a thread attached to the vertex $v_{g/2+k}$ due to Lemma~\ref{lem:cyclebfthread}.

Suppose then that there exists a thread attached to the vertex $v_{g/2+k}$. The set $S[v \leftarrow v_k]$ is clearly a metric basis of $G-v$. Since $b(G-v)=1$, all metric bases of $G-v$ are of the form $S[v \leftarrow u]$, where $S$ is a metric basis of $G$ and $u \in V(T_{v_i})$ for some $i\neq 0$. Denote $m = \min \{ j \geq 1 \ | \ \deg_G(v_j) \geq 3 \text{ or } \deg_G(v_{g/2+j}) \geq 3 \}$. The set $S[v \leftarrow u]$ is not a metric basis if $i \neq k$:
\begin{itemize}
\item If $i \in [1,k-1]$, then the graph $G-v$ contains configuration $\B$ with respect to the set $S[v \leftarrow u]$ due to the thread attached to $v_{g/2+k}$.
\item If $i \in [k+1,g/2-1]$, then the graph $G-v$ contains configuration $\C$ with respect to the set $S[v \leftarrow u]$.
\item If $i = g/2$, then the graph $G-v$ contains configuration $\A$ with respect to the set $S[v \leftarrow u]$.
\item If $i \in [g/2+1,g/2+m]$, then the graph $G-v$ contains configuration $\C$ with respect to the set $S[v \leftarrow u]$, because due to Lemma~\ref{lem:mthread} there exists a thread at some $v_j$, where $j \in [g/2+m+1,g-1]$.
\item If $i \in [g/2+m+1,g-1]$, then the graph $G-v$ contains configuration $\B$ with respect to the set $S[v \leftarrow u]$ due to the definition of $m$.
\end{itemize}
Consequently, $u=v_k$ and $v_k$ is a basis forced vertex of $G-v$.
\end{proof}

\subsection{Characterisation of Unicyclic Graphs With $b(G)=1$ and One Basis Forced Vertex}

In this section we fully characterize the unicyclic graphs and their basis forced vertices when $b(G)=1$. Now we know one spot of the graph where we have elements of a metric basis (in order to have a branch-resolving set). There is only one additional element (due to the metric dimension). The following lemma regarding the vertex that is branch-active.

\begin{lem} \label{Lemma_branch_active_thread}
Let $G$ be a unicyclic graph with $b(G)=1$, and $v$ be branch-active. If $G$ contains a basis forced vertex, then there is no thread attached to $v$.
\end{lem}
\begin{proof}
Let $v_0$ be branch-active, $u$ basis forced, and $S$ a metric basis of $G$. Since $b(G)=1$, we have $\dim (G) = L(G) + 1$ due to Theorem \ref{thm:bfmdim}. Furthermore,
\[ |S \cap V(T_{v_0})| = L(G) = \sum_{v\in V(T_{v_0}), \ell(v) > 1} (\ell(v) - 1). \]
Thus, if there exist a thread or threads at $v_0$, then one of them is $S$-free, and $G$ contains configuration $\B$ (a contradiction).
\end{proof}

Recall that, by Theorem~\ref{thm:basisforced}, a basis forced vertex of a unicyclic graph is either a cycle vertex of degree two or a sole pendant attached to a cycle vertex. In the following theorem, we are now ready to characterise the basis forced vertices in pendants.

\begin{thm} \label{Thm_pendant_char}
	Let $G$ be a unicyclic graph with $b(G)=1$ and let $C$ be its cycle labelled in such a way that $v_0$ is branch-active. Assume further that $v$ is a pendant attached to $v_j \in C$ and $V(T_{v_j}) = \{v_j,v\}$. Now $v$ is a basis forced vertex of $G$ if and only if
	\begin{enumerate}[label={\em(\arabic*)}]
		\item the girth $g$ of $G$ is even,
		\item no thread is attached to $v_0$,
		\item $j \in [2,g/2-1]$,
		\item $\deg (v_i) = 2$ for all $i \in [j+1,g/2-1] \cup [g/2+j+1,g-1]$,
		\item every thread attached to some $v_i$ where $i \in [1,j-1]$ is of length at most $g/2-j-1$,
		\item for $m = \min \{ l \geq 1 \ | \ \deg(v_l) \geq 3 \text{ or } \deg(v_{g/2+l}) \geq 3 \}$ we have $m < j$ and there exists a thread of length at least $m$ at some $v_i$ where $i \in [g/2+m+1,g/2+j]$, and
		\item $j = g/2-2$ or there exists a thread of length $g/2-j-1$ at some $v_i$ where $i \in [1,j-1]$.  	
	\end{enumerate}
\end{thm}
\begin{proof}
	($\Rightarrow$) Assume first that $v$ is a basis forced vertex of $G$. The conditions~(1)--(7) can be shown to be satisfied based on the previously presented results. Indeed, the conditions~(1)--(7) hold by Theorem~\ref{thm:even}, Lemma~\ref{Lemma_branch_active_thread}, Lemma~\ref{lem:klimits}, Lemma~\ref{lem:bfproperties}(iv), Lemma~\ref{lem:bfproperties}(ii), Lemma~\ref{lem:mthread} and Lemma~\ref{lem:bfproperties}(iii), respectively.
	
	($\Leftarrow$) Assume then that the conditions~(1)--(7) hold. Let $R \ (\subseteq V(T_{v_0}))$ be a minimum branch-resolving set of $G$. In what follows, we first show that $S=R \cup \{v\}$ is a resolving set of $G$ due to Theorem \ref{thm:abc}:
	\begin{itemize}
		\item[$\A$:] The graph $G$ does not contain configuration $\A$ with respect to the set $S$ since by~(3) we have $j \in [2,g/2-1]$.
		\item[$\B$:] By~(2), there is no thread attached to $v_0$ and, thus, there is no $S$-free thread at $v_0$. By~(4) and the fact that $V(T_{v_j}) = \{v_j,v\}$, there are no $S$-free threads at $v_j$ or $v_i$ where $i \in [j+1,g/2-1] \cup [g/2+j+1,g-1]$ either. Thus, the graph $G$ does not contain configuration $\B$ with respect to the set $S$.
		\item[$\C$:] By the previous case, there are no $S$-free threads at $v_0$ or $v_j$. Furthermore, by~(5), all the threads attached to some $v_i$ where $i \in [1,j-1]$ are of length at most $g/2-j-1$. Thus the graph $G$ does not contain configuration $\C$ with respect to the set $S$.
	\end{itemize}
	According to Theorem~\ref{thm:unicdim}, we have $\dim (G) = L(G) + 1$. Thus, the set $S$ is a metric basis of $G$.
	
	Observe that all metric bases of $G$ are of the form $R \cup \{u\}$, where $R$ is a minimum branch-resolving set and $u \in V(T_i)$ for some $i \neq 0$. In what follows, we show that if $S'=R \cup \{u\}$ is a metric basis of $G$, then $i=j$:
	\begin{itemize}
		\item $i \in [1,j-1]$: The pendant $v$ at the vertex $v_j$ causes $G$ to contain configuration $\B$ with respect to the set $S'$.
		\item $i \in [j+1,g/2-1]$: Observe first that if $j=g/2-2$, then $i = g/2-1$ and the pendant $v$ causes the graph $G$ to contain configuration $\C$ with respect to the set $S'$. Otherwise, due to the condition~(7), there exists a thread of length $g/2 -j-1 \geq g/2 - i$ at some $v_l$ where $l \in [1,j-1]$. Thus, the graph $G$ contains configuration $\C$ with respect to the set $S'$.
		\item $i = g/2$: The graph $G$ contains configuration $\A$ with respect to the set $S'$.
		\item $i \in [g/2+1,g/2+m]$: Observe first that the distance $d(v_0,v_i) = g-i$. By the condition~(6), there exists a thread of length at least $m \geq g/2-(g-i)$ for some $v_l$ with $l \in [g/2+m+1,g/2+j]$. Hence, $G$ contains configuration~$\C$ with respect to the set $S'$.
		\item $i \in [g/2+m+1,g-1]$: Due to condition 6, there exists an $S'$-free thread at $v_m$ or $v_{g/2+m}$ since $b(G)=1$ and $v_0$ is branch-active. Thus, the graph $G$ contains configuration $\B$ with respect to the set $S'$. 
	\end{itemize}
	Thus, in conclusion, $i=j$. Recall that $V(T_{v_j}) = \{v_j,v\}$. If $u = v_j$, then the vertex $v$ forms an $S'$-free thread and the graph $G$ contains configuration $\B$ with respect to the set $S'$. Therefore, $u=v$ and $v$ is a basis forced vertex of $G$.
\end{proof}

Based on Theorem~\ref{Thm_pendant_char}, we also obtain a characterisation for the basis forced vertices in the cycle as given in the following theorem.
\begin{thm}
	Let $G$ be a unicyclic graph with $b(G)=1$ and $C$ be its cycle labelled in such a way that $v_0$ is branch-active. Assume further that $v_j \in C$ is such that $\deg(v_j) = 2$. Now $v_j$ is a basis forced vertex of $G$ if and only if the conditions~\emph{(1)}--\emph{(6)} of Theorem~\ref{Thm_pendant_char} are met and the following two additional conditions are satisfied:
	\begin{itemize}
		\item[\emph{(7')}] there exists a thread of length $g/2-j-1$ at some $v_i$ where $i \in [1,j-1]$ and
		\item[\emph{(8)}] there exists a thread attached to the vertex $v_{g/2+j}$ (antipodal to $v_j$) .
	\end{itemize}
\begin{proof}
	($\Rightarrow$) Assume first that $v_j$ is a basis forced vertex of $G$. The conditions~(1)--(6) hold as in the case of Theorem~\ref{Thm_pendant_char}. Furthermore, the condition~(7') is satisfied due to Lemma~\ref{lem:bfproperties}(iii) and the condition~(8) follows by Lemma~\ref{lem:cyclebfthread}.
	
	($\Leftarrow$) Assume then that the conditions~(1)--(6), (7') and (8) hold. Suppose to the contrary that $v_j$ is not a basis forced vertex of $G$. Then form a new graph $G'$ from $G$ by adding a pendant $v$ to $v_j$. Now the conditions~(1)--(7) of Theorem~\ref{Thm_pendant_char} are satisfied and $v$ is a basis forced vertex of $G'$. However, by Lemma~\ref{lem:pendanttocycle}, it follows that $v_j$ is basis forced vertex of $G'-v = G$ (a contradiction).
\end{proof}
\end{thm}

\section{Strong Basis Forced Vertices in Unicyclic Graphs}

In order to use the tool of strong resolving graphs, our first comments are addressed to describe how the strong resolving graph of a unicyclic graph $G$ looks like. To this end, we need to divide the study into two cases, depending on the parity of the unique cycle of $G$. We shall use a similar terminology and notation as already commented, where $G$ is a unicyclic graph with the cycle $C = v_0v_1\cdots v_{g-1}v_0$. We first give some observations and basic results.

\begin{observation}
	Let $G$ be a unicyclic graph with the cycle $C = v_0v_1\cdots v_{g-1}v_0$. Then,
	\begin{itemize}
		\item Any two vertices of degree one in $G$ are MMD. Thus, $N_1(G)$ induces a clique in $G_{SR}$.
		\item Any two diametral vertices of $C$ of degree two in $G$ are MMD, and they induce an edge in $G_{SR}$.
		\item If $v_i\in C$ has degree two, then $v_i$ is MMD with all the vertices $u\in \left(N_1(T_{v_{i+\lfloor g/2\rfloor}})\cup N_1(T_{v_{i+\lceil g/2\rceil}})\right)\setminus C$.
		\item Every cut vertex of $G$ is not MMD with any other vertex of $G$. Thus, it is isolated in $G_{SR}$.
	\end{itemize}
\end{observation}

\begin{proposition}\label{pro:constructions}
	Let $G$ be a unicyclic graph with the cycle $C = v_0v_1\cdots v_{g-1}v_0$. Then $G_{SR}$ can be obtained as follows.
	\begin{itemize}
		\item $g$ even: We begin with $g/2$ edges formed by the vertices $v_iv_{i+g/2}$ for every $i\in\{0,\dots, g/2-1\}$. Next, for every $v_j\in C$ of degree at least three in $G$, we substitute $v_j$ by a clique of cardinality $|N_1(T_{v_j})\setminus \{v_j\}|$, and add all possible edges between $v_{j+g/2}$ (or a clique corresponding to it if already added) and the vertices of the added clique. Finally, we add all possible edges between any two vertices belonging to any two different cliques of the added ones in the step above, and the cut vertices of $G$ are added as isolated vertices of $G_{SR}$. Notice that such graph may not be connected, even if we do not consider the isolated vertices.
		\item $g$ odd: We begin with a cycle $C'=v_0v_{(g+1)/2}v_1v_{(g+1)/2+1}v_2v_{(g+1)/2+2}\cdots v_{(g-1)/2}v_0$. Next, for every $v_j\in C$ of degree at least three in $G$, we substitute $v_j$ by a clique of cardinality $|N_1(T_{v_j})\setminus \{v_j\}|$, and add all possible edges between $v_{j+(g-1)/2},v_{j+(g+1)/2}$ (or the cliques corresponding to them if already added) and the vertices of the added clique. Finally, we add all possible edges between any two vertices belonging to any two different cliques of the added ones in the step above, and the cut vertices of $G$ are added as isolated vertices of $G_{SR}$.
	\end{itemize}
\end{proposition}

Based on these two results above, we can easily deduce the following reduction.

\begin{proposition}\label{prop:equal-SRG}
	Let $G$ be a unicyclic graph with the cycle $C = v_0v_1\cdots v_{g-1}v_0$ where $S\subseteq V(C)$ is the set of vertices of $C$ of degree larger than two in $G$. Let $G'$ be a unicyclic graph with a cycle $C' = v'_0v'_1\cdots v'_{g-1}v'_0$, where the set $S'\subseteq V(C')$ of vertices of $C'$ of degree larger than two in $G$ satisfies that for every $v'_i\in S'$,
	\begin{itemize}
		\item $v'_i\in S'$ if and only if $v_i\in S$,
		\item $v'_i$ has $|N_1(T_{v_i})\setminus\{v_i\}|$ adjacent pendant vertices, and
		\item the only cut vertex of $G'$ in $T_{v'_i}$ is $v'_i$.
	\end{itemize}
	Then, the strong resolving graphs of $G$ and of $G'$ differ only on $|V(G)|-|V(G')|$ isolated vertices.
\end{proposition}

An example of a unicyclic graph $G$, its related graph $G'$ (as described in the proposition above), and their strong resolving graph (without the isolated vertices) are given in Figure \ref{fig:G-G'-GSR}.

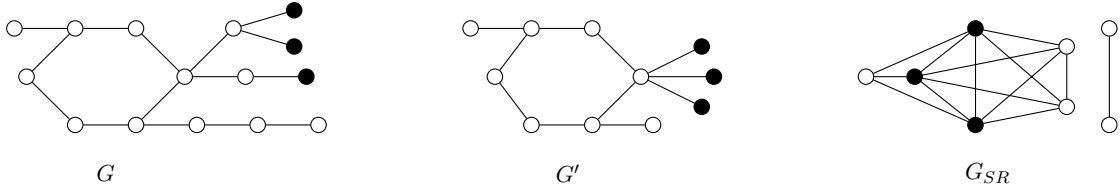
\begin{figure}[ht]
	\centering
	\begin{tikzpicture}[scale=.8, transform shape]
		
		\node [draw, shape=circle, scale=.7] (a0) at  (-0.8,0.8) {};
		\node [draw, shape=circle, scale=.7] (a1) at  (0,1.6) {};
		\node [draw, shape=circle, scale=.7] (a2) at  (1,1.6) {};
		\node [draw, shape=circle, scale=.7] (a3) at  (1.8,0.8) {};
		\node [draw, shape=circle, scale=.7] (a4) at  (1,0) {};
		\node [draw, shape=circle, scale=.7] (a5) at  (0,0) {};
		
		\node [draw, shape=circle, scale=.7] (x1) at  (-1,1.6) {};
		\node [draw, shape=circle, scale=.7] (x3) at  (2.8,0.8) {};
		\node [draw, shape=circle, scale=.7, fill=black] (y3) at  (3.8,0.8) {};
		
		\node [draw, shape=circle, scale=.7] (t3) at  (2.6,1.6) {};
		\node [draw, shape=circle, scale=.7, fill=black] (w3) at  (3.6,1.3) {};
		\node [draw, shape=circle, scale=.7, fill=black] (z3) at  (3.6,1.9) {};
		
		\node [draw, shape=circle, scale=.7] (x4) at  (2,0) {};
		\node [draw, shape=circle, scale=.7] (y4) at  (3,0) {};
		\node [draw, shape=circle, scale=.7] (z4) at  (4,0) {};
		\node at (0.5,-0.8) {$G$};
		
		\draw(a0)--(a1)--(a2)--(a3)--(a4)--(a5)--(a0);
		\draw(a1)--(x1);
		\draw(a3)--(x3)--(y3);
		\draw(a3)--(t3)--(w3);
		\draw(t3)--(z3);
		\draw(a4)--(x4)--(y4)--(z4);

		\node [draw, shape=circle, scale=.7] (ba0) at  (6.9,0.8) {};
		\node [draw, shape=circle, scale=.7] (ba1) at  (7.5,1.6) {};
		\node [draw, shape=circle, scale=.7] (ba2) at  (8.5,1.6) {};
		\node [draw, shape=circle, scale=.7] (ba3) at  (9.3,0.8) {};
		\node [draw, shape=circle, scale=.7] (ba4) at  (8.5,0) {};
		\node [draw, shape=circle, scale=.7] (ba5) at  (7.5,0) {};
		
		\node [draw, shape=circle, scale=.7] (bx1) at  (6.5,1.6) {};
		\node [draw, shape=circle, scale=.7, fill=black] (bx3) at  (10.5,0.8) {};
		
		\node [draw, shape=circle, scale=.7, fill=black] (bt3) at  (10.3,1.3) {};
		
		\node [draw, shape=circle, scale=.7, fill=black] (bt4) at  (10.3,0.3) {};
		
		\node [draw, shape=circle, scale=.7] (bx4) at  (9.5,0) {};
		\node at (8.1,-0.8) {$G'$};
		
		\draw(ba0)--(ba1)--(ba2)--(ba3)--(ba4)--(ba5)--(ba0);
		\draw(ba1)--(bx1);
		\draw(ba3)--(bx3);
		\draw(ba3)--(bt3);
		\draw(ba3)--(bt4);
		\draw(ba4)--(bx4);

		\node [draw, shape=circle, scale=.7] (r4) at  (13,0.8) {};
		\node [draw, shape=circle, scale=.7, fill=black] (r) at  (13.8,0.8) {};
		\node [draw, shape=circle, scale=.7, fill=black] (r0) at  (14.8,1.6) {};
		\node [draw, shape=circle, scale=.7, fill=black] (r1) at  (14.8,0) {};
		\node [draw, shape=circle, scale=.7] (r2) at  (16.3,1.3) {};
		\node [draw, shape=circle, scale=.7] (r3) at  (16.3,0.3) {};
		
		\node [draw, shape=circle, scale=.7] (r5) at  (17,1.6) {};
		\node [draw, shape=circle, scale=.7] (r6) at  (17,0) {};
		\node at (15,-0.8) {$G_{SR}$};
		
		\draw(r0)--(r)--(r1);
		\draw(r2)--(r)--(r3);
		\draw(r1)--(r2)--(r0)--(r1)--(r3)--(r0);
		\draw(r2)--(r3);
		\draw(r5)--(r6);
		\draw(r4)--(r);
		\draw(r1)--(r4)--(r0);
		
	\end{tikzpicture}
	\caption{A unicyclic graph $G$, its related graph $G'$ and $G_{SR}$ {\em (}without isolated ones{\em )}.}\label{fig:G-G'-GSR}
\end{figure}

Once given some properties of the strong resolving graph of a unicyclic graph, we are then able to present our results on strong basis forced vertices of such graphs. For instance, it is easy to observe that the bolded vertices are strong basis forced vertices of the graphs $G$ and $G'$ represented in Figure \ref{fig:G-G'-GSR}, since any vertex cover of minimum cardinality in $G_{SR}$ must contain such vertices, or equivalently, these bolded vertices form the set $VC(G_{SR})$. This means that, in contrast with the classical version of metric dimension, a unicyclic graph can contain more than two strong basis forced vertices. Indeed, as we next show, a unicyclic graph can contain as many strong basis forced vertices as we would require.

Let $n\ge 2$ be an integer. We construct a unicyclic graph $G_n$ as follows. We begin with a cycle $C_{2n+2}$. Then, to obtain $G_n$, we add a pendant vertex to exactly $n+2$ consecutive vertices of the cycle $C_{2n+2}$. Now, the strong resolving graph $(G_n)_{SR}$ is isomorphic to a complete graph $K_{n+2}$ with one pendant vertex added to $n$ of its vertices, together with $n+2$ isolated vertices. It can be noted that every vertex cover set of minimum cardinality in $(G_n)_{SR}$ contains the $n$ vertices of $K_{n+2}$ having a pendant vertex as a neighbor, and exactly one vertex of the remaining two vertices of $K_{n+2}$ not having an adjacent pendant vertex. Therefore, these $n$ vertices mentioned before form precisely the set $VC(G_{SR})$, and so, they are strong basis forced vertices of $G$.

On the other hand, based on Proposition \ref{prop:equal-SRG}, we observe that strong basis forced vertices of a unicyclic graph $G$ with cycle $C = v_0v_1\cdots v_{g-1}v_0$ are independent of the structure of the components $T_{v_i}$ for every $v_i$ of $C$. In this sense, from now on we shall only consider unicyclic graphs $G$ having a structure as the graphs $G'$ described in Proposition \ref{prop:equal-SRG}. In the following lemma, we present a couple of useful results.
\begin{lemma}\label{lem:basic-prop}
	Let $G$ be a unicyclic graph with the cycle $C = v_0v_1\cdots v_{g-1}v_0$.
	\begin{itemize}
		\item[{\em (i)}] If every vertex of $C$ has degree at least three, then $G$ has no strong basis forced vertices.
		\item[{\em (ii)}] If there is at most one vertex $v_i$ of $C$ of degree at least three, then $G$ has no strong basis forced vertices.
	\end{itemize}\begin{proof}
	(i) It is clear from the constructions given in Proposition \ref{pro:constructions} that the strong resolving graph $G_{SR}$ has a component isomorphic to a complete graph $K_{n_1(G)}$ and $|V(G)|-n_1(G)$ isolated vertices. Hence, we note that $corona(G_{SR})=V(G_{SR})$, and so $VC(G_{SR})=\emptyset$, which leads to our conclusion.
	
	(ii) If every vertex of $C$ has degree two, then $G$ is a cycle, which clearly has no strong basis forced vertices. Assume now that $G$ has one vertex, say $v_i$, of $C$ of degree at least three. If $g$ is even, then by the construction given in Proposition \ref{pro:constructions}, the graph $G_{SR}$ has $(g-2)/2$ components isomorphic to $K_2$, one trivial component $K_1$, and one component isomorphic to a clique. Hence, we observe that $corona(G_{SR})=V(G_{SR})$, and so $VC(G_{SR})=\emptyset$.

Now, if $g$ is odd, then again by the construction given in Proposition \ref{pro:constructions}, we have that $G_{SR}$ has one trivial component $K_1$, and one component isomorphic to a cycle in which one of its vertices, say $x$, has been substituted by a clique, with all vertices of this clique being adjacent to the two neighbors of $x$ in such cycle. Thus, we again deduce that $VC(G_{SR})=\emptyset$, which altogether allow to conclude that $G$ has no strong basis forced vertices.
\end{proof}\end{lemma}

In what follows, we divide the study into two sections based on the parity of the cycle. 

\subsection{Unicyclic graphs with an even cycle}

For the rest of the section, we assume that girth $g$ of the unicyclic graph $G$ is even. In the following lemma, we first present some basic properties on the strong basis forced vertices of such graphs.
\begin{lemma}\label{lem:basic-prop_even}
	Let $G$ be a unicyclic graph such that its girth $g$ is even and the cycle $C = v_0v_1\cdots v_{g-1}v_0$.
	\begin{itemize}
		\item[{\em (i)}] If the diametral vertices $v_i$ and $v_{i+g/2}$ are of degree two, then $v_i$ and $v_{i+g/2}$ are not strong basis forced vertices.
		\item[{\em (ii)}] If for all the pairs of diametral vertices $v_i$ and $v_{i+g/2}$ at least one of them has degree two in $G$, then $G$ has no strong basis forced vertices.
	\end{itemize}
\end{lemma}

\begin{proof}
%
	
	(i) If $v_i$ and $v_{i+g/2}$ have degree two, then clearly $v_i$ and $v_{i+g/2}$ induce a component in $G_{SR}$ isomorphic to $K_2$. Thus, neither of them belongs to $VC(G_{SR})$, and so they are not strong basis forced vertices.
	
	(ii) If $v_i$ and $v_{i+g/2}$ are both of degree two, then by (i), they induce a component in $G_{SR}$ isomorphic to $K_2$ and are not strong basis forced vertices. Assume now that exactly one of the vertices $v_i$ or $v_{i+g/2}$ is of degree two in $G$. From the constructions given in Proposition \ref{pro:constructions}, the strong resolving graph $G_{SR}$ is isomorphic to a graph having a vertex partition formed by three sets $X,Y,Z$, where $X$ induces a clique (all the leaves in $G$), $Y\cup Z$ induces an independent set ($Y$ is the set of vertices of $C$ of degree two and $Z$ is the set of cut vertices of $G$), the neighborhoods of vertices of $Y$ in $X$ form a vertex partition of $X$, and the vertices of $Z$ are isolated. Thus, one can readily observe that the independence number of such graph equals $g/2$ and that $VC(G_{SR})=\emptyset$ since $corona(G_{SR})=V(G_{SR})$. Therefore, there is no strong basis forced vertex as well.
\end{proof}

From now on, in order to give the number of strong basis forced vertices of unicyclic graphs we shall describe some notations and definitions. Based on the results listed above, it is enough to consider unicyclic graphs with structure as the graph $G'$ described in Proposition \ref{prop:equal-SRG}. Consider $G$ is a unicyclic graph with cycle $C = v_0v_1\cdots v_{g-1}v_0$ with $g$ even. For a vertex $v_i$ of $C$, by $Q_{v_i}$ we denote the set of vertices of degree one in $G$ adjacent to $v_{i+g/2}$ (if $v_{i+g/2}$ has degree two, then $Q_{v_i}$ is an empty set). By $D_2(C)$ we denote the set of vertices from the pairs $v_j,v_{j+g/2}$ such that both vertices have degree two, and let $d_2(C)=|D_2(C)|/2$. Also, by $D_{>2}(C)$ we denote the set of vertices from the pairs $v_j,v_{j+g/2}$ such that both vertices have degree larger than two, and let $d_{>2}(C)=|D_{>2}(C)|/2$. Notice that if $D_{>2}(C)=\emptyset$ or $D_{>2}(C)=C$, then $G$ has no strong basis forced vertices, according to Lemma~\ref{lem:basic-prop}(i) and Lemma~\ref{lem:basic-prop_even}(ii). Thus, from now on we assume that $D_{>2}(C)\ne\emptyset$ and that $D_{>2}(C)\ne C$ for any graph $G$. With these notations, we note the following fact.

\begin{lemma}\label{lem:indep-GSR-even}
Let $G$ be a unicyclic graph such that its girth $g$ is even, the cycle $C = v_0v_1\cdots v_{g-1}v_0$ and $D_{>2}(C)\ne\emptyset$. Then a maximum independent set of $G_{SR}$ consist of \emph{(i)} the vertices $v_i \in C$ with $\deg_G(v_i) \geq 3$, denoted by $X$, \emph{(ii)} one vertex of each pair $v_j,v_{j+g/2}\in D_2(C)$, (iii) the set of vertices $v_i$ of degree two in $G$ such that $Q_{v_i}\ne \emptyset$, and (iv) one vertex of the set $\bigcup_{v_i\in D_{>2}(C)}Q_{v_i}$. Thus, in total, the cardinality of a maximum independent set of $G_{SR}$ is at least $|X|+g/2-d_{>2}(C)+1$.
\end{lemma}

As a consequence of the lemma, one can see that the set $corona(G_{SR})$ contains the vertices $v_i \in C$ with $\deg_G(v_i) \geq 3$, both vertices of the pairs $v_j,v_{j+g/2}\in D_2(C)$, the set of vertices $v_i$ of degree two in $G$ such that $Q_{v_i}\ne \emptyset$, and every vertex of the set $\bigcup_{v_i\in D_{>2}(C)}Q_{v_i}$.

\begin{theorem}
	Let $G$ be a unicyclic graph with the cycle $C = v_0v_1\cdots v_{g-1}v_0$ with $g$ even.
	\begin{itemize}
		\item[\emph{(i)}] If $D_{>2}(C)=\emptyset$, then there are no strong basis forced vertices of $G$.
		\item[\emph{(ii)}] If $D_{>2}(C)\neq\emptyset$, then the strong basis forced vertices of $G$ are the vertices of the nonempty sets $Q_{v_i}$. Hence, in total, the number of strong basis forced vertices is $\sum_{v_i\notin D_2(C)\cup D_{>2}(C)}|Q_{v_i}|$.
	\end{itemize}
\end{theorem}

\begin{proof}
	(i) The first claim immediately follows by Lemma~\ref{lem:basic-prop_even}(ii) (as discussed earlier). (ii) Assuming  $D_{>2}(C)\neq\emptyset$, the second claim straightforwardly follows by Lemma~\ref{lem:indep-GSR-even} (and the observations afterwards).
\end{proof}

One conclusion that we can deduce from the result above is that the unicyclic graphs $G$ whose unique cycle has even order can have only strong basis forced vertices which are vertices of degree one. In contrast with this situation, as shown in the next section, for the case of unicyclic graphs $G$ whose unique cycle has odd order, there could be strong basis forced vertices which are vertices of degree one as well as vertices of the cycle; see Theorem~\ref{Thm_Char_odd_cycle} and Example~\ref{Ex_odd_cycle}. 

\subsection{Unicyclic graphs with an odd cycle}

In this section, we characterize the strong basis forced vertices in the unicyclic graphs $G$ of which unique cycle $C = v_0v_1\cdots v_{g-1}v_0$ has an odd order $g$. To this end, we first introduce some notation and terminology. Recall that the vertices $u$ and $v$ in $C$ are MMD in $G$ if and only if $\deg_G(u) = \deg_G(v) = 2$ and $d_G(u,v) = (g-1)/2$. Assuming $i,j\in \{0,\dots,g-1\}$, let $Q_{[i,j]}$ denote the maximal sequence of vertices $v_{i}v_{i+(g-1)/2}v_{i+1}v_{i+1+(g-1)/2}\cdots v_j$ in $G_{SR}$ such that the degree of each vertex in $G$ is equal to $2$ and each pair of consecutive vertices are MMD, \emph{i.e.}, adjacent in $G_{SR}$. Hence, $j$ is either equal to $i+k$ or $i+k+(g-1)/2$ for some $k \geq 0$. We call $v_i$ and $v_j$ the \emph{end-vertices} of $Q_{[i,j]}$.
Furthermore, we denote $q_{[i,j]}=|Q_{[i,j]}|$ and say that the sequence $Q_{[i,j]}$ is of \emph{odd length} (or simply \emph{odd}) if $2 \nmid q_{[i,j]}$, and otherwise it is of \emph{even length} (or simply \emph{even}). Denote the set of every other vertex of $Q_{[i,j]}$ starting from the first one by $A_1(Q_{[i,j]})$ and the set of every other vertex of $Q_{[i,j]}$ starting from the second one by $A_2(Q_{[i,j]})$. Notice that $|A_1(Q_{[i,j]})| = \lceil q_{[i,j]}/2 \rceil$ and $|A_2(Q_{[i,j]})| = \lfloor q_{[i,j]}/2 \rfloor$.

Analogously to the sets $Q_{v_i}$ in the case of unicyclic graphs with an odd cycle, we now define $Q'_{v_i}$ and $Q''_{v_i}$ to be the sets of vertices of degree one in $G$ adjacent to $v_{i+\lfloor g/2\rfloor}$ and $v_{i+\lceil g/2\rceil}$, respectively. Observe that the vertices of $Q'_{v_i}$ and $Q''_{v_i}$ are MMD with $v_i$ or its leaves (if such vertices exist) and that if $v_{i+\lfloor g/2\rfloor}$ or $v_{i+\lceil g/2\rceil}$ has degree two in $G$, then $Q'_{v_i}$ or $Q''_{v_i}$, respectively, is empty. Furthermore, if $Q_{[i,j]}$ is a maximal sequence of vertices as defined above, then $Q'_{v_i}$ or $Q''_{v_i}$ is non-empty, as well as, $Q'_{v_j}$ or $Q''_{v_j}$ is non-empty. We may now consider the graph $G_{SR}$ to be formed as follows:
\begin{itemize}
	\item The vertices of each maximal sequence $Q_{[i,j]}$ induce a path in $G_{SR}$.
	\item The leaves (or pendants) of $G$ induce a clique in $G_{SR}$.
	\item Each leaf $u$ of $G$ is adjacent to $v_i \in C$ in $G_{SR}$ if and only if $\deg_G(v_i) = 2$ and $u \in Q'_{v_i} \cup Q''_{v_i}$.
\end{itemize}

In the following theorem, we give a characterization of the strong basis forced vertices in the case of an odd cycle $C$. Notice that in the cases~(i) and (iii) we might have $corona(G_{SR})=V(G_{SR})$ implying $G$ has no strong basis forced vertices.

\begin{theorem} \label{Thm_Char_odd_cycle}
	Let $U$ be a subset of vertices $u \in C$ such that $\deg(u) \geq 3$ and the leaves of $u$ in $G$ are not adjacent to an end-vertex of a maximal sequence $Q_{[i,j]}$ of odd length in $G_{SR}$. The strong basis forced vertices of $G$ can be determined based on $U$ as follows:
	\begin{itemize}
		\item[\emph{(i)}] If $|U| \geq 2$, or $|U| = 1$ and the unique vertex $u \in U$ is such that the leaves of $u$ in $G$ are not adjacent to any maximal sequence in $G_{SR}$, then the strong basis forced vertices of $G$ are the leaves of $G$ adjacent to a maximal sequence of odd length in $G_{SR}$ and the vertices $A_2(Q_{[i,j]})$ for each $Q_{[i,j]}$ of odd length.
		\item[\emph{(ii)}] If $|U| = 1$ and the unique vertex $u \in U$ is such that the leaves of $u$ in $G$ are adjacent to a maximal sequence of even length in $G_{SR}$, then the strong basis forced vertices of $G$ are the leaves adjacent to a maximal sequence of odd length in $G_{SR}$, the vertices $A_2(Q_{[i,j]})$ for each $Q_{[i,j]}$ of odd length and the vertices $A_k(Q_{[i,j]})$ of a maximal sequence of even length adjacent in $G_{SR}$ to the leaves of $u$ in $G$, where $k$ is chosen in such a way that the vertex of $Q_{[i,j]}$ adjacent in $G_{SR}$ to the leaves of $u$ in $G$ belongs to $A_k(Q_{[i,j]})$.
		\item[\emph{(iii)}] If $U = \emptyset$, then the strong basis forced vertices of $G$ are the leaves of $G$ adjacent to two maximal sequences of odd length in $G_{SR}$, denoted by $X$, and the vertices $A_2(Q_{[i,j]})$ for each odd maximal sequence $Q_{[i,j]}$ with its end-vertices adjacent to two vertices $x_1$ and $x_2$ of $X$ in $G_{SR}$ such that $x_1$ and $x_2$ are not leaves of a same vertex of $C$ in $G$.
	\end{itemize}
\end{theorem}
\begin{proof}
	Observe first that if $U \neq \emptyset$, then the cardinality of a maximum independent set of $G_{SR}$ is equal to
	\[
	1 + \sum_{\ell=1}^t \lceil q_{[i_{\ell},j_{\ell}]} / 2 \rceil \text,
	\]
	where $t$ denotes the number of maximal sequences $Q_{[i,j]}$. Indeed, a maximum independent set of $G_{SR}$ can be formed by choosing a leaf of a vertex belonging $U$ and the $\lceil q_{[i_{\ell},j_{\ell}]}/2 \rceil$ suitable (every other) vertices of each $Q_{[i,j]}$.
	
	(i) Assume that $|U| \geq 2$ or that $|U| = 1$ and the unique vertex $u \in U$ is such that the leaves of $u$ in $G$ are not adjacent to any maximal sequence in $G_{SR}$. It is straightforward to verify the set $corona(G_{SR})$ consist of the isolated vertices of $G_{SR}$, the leaves of vertices of $U$ in $G$, the vertices in the maximal sequences of even length and the vertices $A_1(Q_{[i,j]})$ of each $Q_{[i,j]}$ of odd length. Therefore, as $VC(G_{SR}) = V(G_{SR}) \setminus corona(G_{SR})$, the claim immediately follows.
	
	(ii) Assume that $|U| = 1$ and the unique vertex $u \in U$ is such that the leaves of $u$ in $G$ are adjacent to a maximal sequence of even length in $G_{SR}$. It can be easily seen that $corona(G_{SR})$ consist of the isolated vertices of $G_{SR}$, the leaves of $u$ in $G$, the vertices $A_1(Q_{[i,j]})$ of each $Q_{[i,j]}$ of odd length and the vertices $A_{3-k}(Q_{[i',j']})$ of a maximal sequence $Q_{[i',j']}$ of even length adjacent (in $G_{SR}$) to the leaves of $u$ (in $G$), where $k$ is chosen in such a way that the vertex of $Q_{[i',j']}$ adjacent (in $G_{SR}$) to the leaves of $u$ (in $G$) does not belong to $A_{3-k}(Q_{[i',j']})$. Therefore, as $VC(G_{SR}) = V(G_{SR}) \setminus corona(G_{SR})$, the claim immediately follows.
	
	(iii) Finally, assume that $U = \emptyset$. Now the cardinality of a maximum independent set of $G_{SR}$ is equal to
	\begin{equation} \label{Eq_size_max_independent}
		\sum_{\ell=1}^t \lceil q_{[i_{\ell},j_{\ell}]}/2 \rceil
	\end{equation}
	since each leaf of $G$ is adjacent to a maximal sequence of odd length in $G_{SR}$. Denote by $X$ the leaves of $G$ which are adjacent to two maximal sequences of odd length in $G_{SR}$. Clearly, no vertex of $X$ belongs to $corona(G_{SR})$ since otherwise the size given in~\eqref{Eq_size_max_independent} cannot be reached. Suppose then that $Q_{[i,j]}$ of odd length is such that its end-vertices are adjacent to two vertices $x_1$ and $x_2$ of $X$ in $G_{SR}$ with $x_1$ and $x_2$ not being leaves of a same vertex of $C$ in $G$. Applying a similar argument as above, it can be deduced that the vertices $A_2(Q_{[i,j]})$ do not belong to $corona(G_{SR})$. However, all the other vertices except the mentioned ones belong to $corona(G_{SR})$. Thus, the claim follows.
\end{proof}

Before the previous theorem, it was briefly mentioned that in some cases no strong basis forced vertices might occur. In the following straightforward corollary, this possibility is further discussed.
\begin{corollary}
	Let $U$ be a subset of vertices $u \in C$ such that $\deg(u) \geq 3$ and the leaves of $u$ in $G$ are not adjacent to an end-vertex of a maximal sequence $Q_{[i,j]}$ of odd length in $G_{SR}$.
	\begin{itemize}
		\item[\emph{(i)}] Assume that $|U| \geq 2$, or $|U| = 1$ and the unique vertex $u \in U$ is such that the leaves of $u$ in $G$ are not adjacent to any maximal sequence in $G_{SR}$. All the maximal sequences are of even length in $G_{SR}$ if and only if there are no strong basis forced vertices in $G$.
		\item[\emph{(ii)}] Assume that $U = \emptyset$. There are no leaves of $G$ adjacent to two maximal sequences of odd length in $G_{SR}$ if and only if there are no strong basis forced vertices in $G$.
	\end{itemize}
\end{corollary}

\begin{proof}
	The results immediately follow by Theorem~\ref{Thm_Char_odd_cycle}.
\end{proof}

Observe that although it is not easy to give a simple closed formula for the number of strong basis forced vertices of a unicyclic graph with an odd cycle, the number can be straightforwardly computed based on Theorem~\ref{Thm_Char_odd_cycle}. Recall that in the case of a unicyclic graph with an even cycle, only the leaves of $G$ can be strong basis forced vertices. In the following example, we illustrate the fact that this is not the case with an odd cycle.
\begin{example} \label{Ex_odd_cycle}
	Let $G_{t,q}$ be a unicyclic graph with an odd cycle $C_{2t+1}=v_0v_1\cdots v_{2t}v_0$ and $q$ (distinct) pendants added to each of the vertices $v_0$, $v_t$, $v_{t+1}$ and $v_{2t}$, where $t$ and $q$ are integers such that $t, q \geq 2$. Now the strong resolving graph of $G_{t,q}$ consist of a unique maximal sequence $Q_{[1,t-1]}$ with odd length $2t-3$, the isolated vertices $v_0$, $v_t$, $v_{t+1}$ and $v_{2t}$, and a clique of the $4q$ leaves of $G_{t,q}$. Furthermore, the set $U$ of Theorem~\ref{Thm_Char_odd_cycle} consist of the vertices $v_0$ and $v_t$. Therefore, by Theorem~\ref{Thm_Char_odd_cycle}(i), the strong basis forced vertices of $G_{t,q}$ are the leaves of $v_{t+1}$ and $v_{2t}$ as well as the vertices of $A_2(Q_{[1,t-1]})$. Hence, in total, the number of strong basis forced vertices of $G_{t,q}$ is $\lfloor q_{[1,t-1]}/2 \rfloor + 2q = t-2+2q$. In particular, we can notice that the number of strong basis forced vertices in the leaves and in the cycle can be arbitrarily large.
\end{example}

\section{Concluding remarks}

Since it is already known that a unicyclic graph can have at most 2 basis forced vertices, this work has centered the attention into classifying unicyclic graphs according to the number of basis forced vertices they have. Those unicyclic graphs with number of branch-active
vertices on the cycle equal 1 (that is $b(G)=1$) have been completely dealt with. In consequence, since the unicyclic graphs having basis forced vertices satisfy that $b(G)\in \{0,1\}$, it remains an open problem of considering unicyclic graphs where $b(G)=0$. Some other interesting problems concerning basis forced vertices are as follows.
\begin{itemize}
  \item Are there some graph classes in which the problem of deciding whether a given vertex is a basis forced vertex will be polynomial?
  \item Can we determine the number of basis forced vertices in some simple superclasses of unicyclic graphs like cactus graphs for instance?
\end{itemize}

On the other hand, strong basis forced vertices in unicyclic graphs have been introduced and studied in this work. In this sense, it would be of clear interest to generalize the study of strong basis forced vertices to general graphs. In particular, the following questions could be of interest.
\begin{itemize}
  \item Which is the complexity of deciding whether a given vertex of a graph is a strong basis forced vertex?
  \item Can we compute or bound the number of strong basis forced vertices of general graph?
  \item Can we characterize the class of graphs having strong basis forced vertices?
  \item It would be also interesting to know that graphs that have a unique strong metric basis.
\end{itemize}

\section*{Acknowledgements}

Anni Hakanen has been financially supported by Jenny and Antti Wihuri foundation and the ANR project GRALMECO (ANR-21-CE48-0004-01). Anni Hakanen, Ville Junnila and Tero Laihonen have been partially supported by Academy of Finland grant number 338797. Ismael G.Yero has been partially supported by the Spanish Ministry of Science and Innovation through the grant PID2019-105824GB-I00.


\end{document}